\documentclass[12pt]{amsart}

\usepackage{amssymb, amsfonts, amsmath, amscd, amsthm}
\usepackage{hyperref, eucal, enumerate, mathrsfs}
\usepackage[all]{xy}

\newtheorem{lemma}[subsection]{Lemma}
\newtheorem{proposition}[subsection]{Proposition}
\newtheorem{theorem}[subsection]{Theorem}

\theoremstyle{definition}
\newtheorem{pg}[subsection]{}
\newtheorem{definition}[subsection]{Definition}
\newtheorem{remark}[subsection]{Remark}
\newtheorem{example}[subsection]{Example}

\DeclareMathOperator{\Alg}{Alg}
\DeclareMathOperator{\bbE}{\mathbb{E}}
\DeclareMathOperator{\Br}{Br}
\DeclareMathOperator{\CAlg}{CAlg}
\DeclareMathOperator{\Cat}{\mathcal{C}at}
\newcommand{\Cech}{\v{C}ech\,}
\DeclareMathOperator{\cg}{cg}
\DeclareMathOperator{\cn}{cn}
\DeclareMathOperator*{\colim}{colim}
\DeclareMathOperator{\DerDM}{DerDM}
\DeclareMathOperator{\End}{End}
\DeclareMathOperator{\Ext}{Ext}
\DeclareMathOperator{\fib}{fib}
\DeclareMathOperator{\fpqc}{fpqc}
\DeclareMathOperator{\fsm}{fsm}
\DeclareMathOperator{\Fun}{Fun}
\DeclareMathOperator{\gp}{gp}
\DeclareMathOperator{\Groth}{\mathsf{Groth}}
\DeclareMathOperator{\HH}{H}
\DeclareMathOperator{\Hilb}{Hilb}
\DeclareMathOperator{\Ind}{Ind}
\DeclareMathOperator{\lex}{lex}
\DeclareMathOperator{\LinCat}{LinCat}
\DeclareMathOperator{\LPres}{\mathcal{P}r^{L}}
\DeclareMathOperator{\Map}{Map}
\DeclareMathOperator{\Mod}{Mod}
\DeclareMathOperator{\op}{op}
\DeclareMathOperator{\Pres}{\mathcal{P}r}
\DeclareMathOperator{\PSt}{PSt}
\DeclareMathOperator{\QCoh}{QCoh}
\DeclareMathOperator{\QStk}{QStk}
\DeclareMathOperator{\rev}{rev}
\DeclareMathOperator{\RMod}{RMod}
\DeclareMathOperator{\sBr}{\mathscr{B}r}
\DeclareMathOperator{\Shv}{\mathcal{S}hv}
\DeclareMathOperator{\sm}{sm}
\DeclareMathOperator{\SpDM}{SpDM}
\DeclareMathOperator{\Spec}{Spec}
\DeclareMathOperator{\sPic}{\mathscr{P}ic}
\DeclareMathOperator{\SSet}{\mathcal{S}}
\DeclareMathOperator{\St}{St}
\DeclareMathOperator{\Supp}{Supp}

\DeclareMathOperator{\calA}{\mathcal{A}}
\DeclareMathOperator{\calC}{\mathcal{C}}
\DeclareMathOperator{\calD}{\mathcal{D}}
\DeclareMathOperator{\calE}{\mathcal{E}}
\DeclareMathOperator{\calO}{\mathcal{O}}
\DeclareMathOperator{\calQ}{\mathcal{Q}}
\DeclareMathOperator{\calU}{\mathcal{U}} 
\DeclareMathOperator{\calW}{\mathcal{W}} 
\DeclareMathOperator{\calX}{\mathcal{X}}

\DeclareMathOperator{\sfK}{\mathsf{K}}
\DeclareMathOperator{\sfU}{\mathsf{U}}
\DeclareMathOperator{\sfV}{\mathsf{V}}
\DeclareMathOperator{\sfX}{\mathsf{X}}
\DeclareMathOperator{\sfY}{\mathsf{Y}}

\newcommand{\et}{\mathrm{\acute et}}

\newcommand{\Adjoint}[4]{\xymatrix@1{#1:#2 \ar@<.4ex>[r] & #3:#4 \ar@<.4ex>[l]}}
\newcommand{\Pull}[8]{\xymatrix{#1\ar[r]^-{#5}\ar[d]^-{#6} & #2 \ar[d]^-{#7} \\ #3 \ar[r]^-{#8} & #4}}

\title{Brauer Spaces of Spectral Algebraic Stacks}
\author{Chang-Yeon Chough}
\address{Center for Geometry and Physics, Institute for Basic Science (IBS), Pohang 37673, Republic of Korea}
\email{chough@ibs.re.kr}

\begin{document}

\begin{abstract}
	We study the question of whether the Brauer group is isomorphic to the cohomological one in spectral algebraic geometry. For this, we prove the compact generation of the derived category of twisted sheaves for quasi-compact spectral algebraic stacks with quasi-affine diagonal, which admit a quasi-finite presentation; in particular, we obtain the compact generation of the unbounded derived category of quasi-coherent sheaves and the existence of compact perfect complexes with prescribed support for such stacks. We also study the relationship between derived and spectral algebraic stacks, so that our results can be extended to the setting of derived algebraic geometry.
\end{abstract}

\setcounter{tocdepth}{1} 
\maketitle
\tableofcontents

\section{Introduction}

\begin{pg}
	The purpose of this paper is to formulate and prove the following question addressed by Grothendieck in the setting of spectral algebraic geometry: given a scheme $X$, is the canonical map $\delta:\Br(X) \rightarrow \HH^2_{\et}(X, \mathbb{G}_m)$ an isomorphism of abelian groups? Recall from \cite[1.2]{MR1608798} that an Azumaya algebra is a sheaf of $\calO_X$-algebra which is locally isomorphic to a matrix algebra and that the Brauer group $\Br(X)$ classifies Azumaya algebras up to Morita equivalence. The map $\delta$ is injective, and its image lies in the torsion subgroup if $X$ is quasi-compact by virtue of \cite[1.4, p.205]{MR1608798}. Note that $\delta$ needs not be an isomorphism in general. In the positive direction, we have many important results which include that $\delta$ is surjective for noetherian schemes $X$ of dimension $\leq 1$ or of dimension $2$ when $X$ is also regular by Grothendieck \cite[2.2]{MR1608805} and that the image of $\delta$ is the torsion subgroup of $\HH^2_{\et}(X, \mathbb{G}_m)$ for quasi-compact and separated schemes which admit an ample line bundle by Gabber (and the proof of de Jong \cite[1.1]{deJong}). To\"en established a more general result for quasi-compact and quasi-separated (derived) schemes by introducing the notion of \emph{derived Azumaya algebra} of \cite[2.1]{MR2957304}, which is a ``dg-enhancement of Azumaya algebras"; see \cite[5.1]{MR2957304}. In fact, To\"en addressed the question regarding the map $\delta$ by investigating the compact generation of $\alpha$-twisted derived dg-category $\mathrm{L}_\alpha(X)$ (see \cite[4.1]{MR2957304}) for each element $\alpha \in \HH^2_{\et}(X, \mathbb{G}_m)$: the endomorphism algebra of a compact generator of $\mathrm{L}_\alpha(X)$ is a derived Azumaya algebra whose associated element of $\HH^2_{\et}(X, \mathbb{G}_m)$ is $\alpha$; see \cite[4.6]{MR2957304}. Extending this idea to the spectral setting, Antieau--Gepner obtained a similar result for quasi-compact quasi-separated spectral schemes \cite[7.2]{MR3190610}. This paper originated from the desire to extend these results to algebraic stacks in the derived and spectral settings.

The main result of this paper is the following:
\end{pg}

\begin{theorem}\label{Brauer spaces and Azumaya algebras for quasi-geometric spectral algebraic stacks}
	Let $X$ be a quasi-geometric spectral algebraic stack which admits a quasi-finite presentation. Then each element of $\Br^\dagger(X)$ has the form $[\calA]$ for some Azumaya algebra $\calA$ on $X$. 
\end{theorem}

\begin{remark}\label{a quasi-finite presentation}
	Here the quasi-geometric spectral algebraic stacks of \ref{quasi-geometric spectral algebraic stack} are a formulation of quasi-compact algebraic stacks with quasi-affine diagonal in spectral algebraic geometry. We will say that a quasi-geometric spectral algebraic stack $X$ admits a \emph{quasi-finite presentation} if there exist an $\bbE_\infty$-ring $A$ and a morphism $\Spec A \rightarrow X$ which is locally quasi-finite, faithfully flat, and locally almost of finite presentation (see \cite[4.2.0.1]{lurie2018sag}). The \emph{extended Brauer group} $\Br^\dagger(X)$ is the set of connected components of the \emph{extended Brauer space} $\sBr^\dagger(X)$ of \cite[11.5.2.1]{lurie2018sag}. Given an Azumaya algebra $\calA$ of \cite[11.5.3.7]{lurie2018sag}, the \emph{extended Brauer class} $[\calA]$ is defined as in \cite[11.5.3.9]{lurie2018sag}.  	
\end{remark}

\begin{remark} 
	\ref{Brauer spaces and Azumaya algebras for quasi-geometric spectral algebraic stacks} recovers the following results from classical, derived, and spectral algebraic geometry:
\begin{enumerate}[(i)]
\item The cases of quasi-compact quasi-separated spectral schemes and spectral algebraic spaces by Antieau--Gepner \cite[7.2]{MR3190610} and Lurie \cite[11.5.3.10]{lurie2018sag}, respectively. Also, the case of quasi-compact spectral Deligne-Mumford stacks with quasi-affine diagonal by Hall--Rydh \cite[9.4]{MR3705292}. 

\item Let $R$ be a commutative ring. Then \ref{Brauer spaces and Azumaya algebras for quasi-geometric spectral algebraic stacks} can be applied to the underlying quasi-geometric spectral algebraic stacks of quasi-geometric derived algebraic stacks over $R$ which admit a quasi-finite presentation (see \ref{underlying quasi-geometric spectral algebraic stacks}). In particular, we obtain the cases of ordinary quasi-compact algebraic stacks with quasi-finite and separated diagonal, and of quasi-compact derived Deligne-Mumford stacks with quasi-affine diagonal by Hall--Rydh \cite[9.3, 9.4]{MR3705292} (hence of quasi-compact quasi-separated derived schemes by To\"en \cite[5.1]{MR2957304}).
\end{enumerate}
\end{remark}

\begin{example}
	Let $p$ be a prime number. Let $\mathbb{F}_p$ denote a finite field of order $p$, and let $\mu_p$ denote the ordinary group scheme of $p$-th roots of unity over $\Spec \mathbb{F}_p$. Let $X$ be a quasi-affine spectral Deligne-Mumford stack over $\mathbb{F}_p$. Then \ref{Brauer spaces and Azumaya algebras for quasi-geometric spectral algebraic stacks} can be applied to the classifying stack of $\mu_p$ over $X$, in which case our result is new.
\end{example}

\begin{pg}
	Our approach to \ref{Brauer spaces and Azumaya algebras for quasi-geometric spectral algebraic stacks} is based on Lurie's reformulation of the work of Antieau--Gepner in terms of the theory of quasi-coherent stacks developed in \cite[10.1.1]{lurie2018sag}. Given a quasi-geometric spectral algebraic stack $X$ and an object $\calC$ of the $\infty$-category $\QStk^{\PSt}(X)$ of prestable quasi-coherent stacks (see \cite[10.1.2.4]{lurie2018sag}), the $\infty$-category $\QCoh(X; \calC)$ of global sections (see \cite[10.4.1.1]{lurie2018sag}) is a spectral analogue of the twisted derived dg-category of To\"en. This perspective leads to the following central definition of interest:
\end{pg}

\begin{definition}\label{twisted compact generation}
	A quasi-geometric spectral algebraic stack $X$ is of \emph{twisted compact generation} if it satisfies the following condition:
\begin{itemize}
\item[$(\ast)$] For each compactly generated stable quasi-coherent stack $\calC$ on $X$, the stable $\infty$-category $\QCoh(X; \calC)$ is compactly generated.
\end{itemize}
\end{definition}

\begin{remark}\label{beta-Thomason condition}
	Our definition of twisted compact generation \ref{twisted compact generation} is related to the $\beta$-Thomason condition on ordinary algebraic stacks of \cite[8.1]{MR3705292}. Let $\beta$ be a regular cardinal and consider the following straightforward generalization which we refer to as the $\beta$-Thomason condition on a quasi-geometric spectral algebraic stack $X$:
\begin{enumerate}[(i)]
\item The $\infty$-category $\QCoh(X)$ is compactly generated by a set of cardinality at most $\beta$. 
\item For every quasi-compact open subset $\calU\subseteq |X|$, there exists a compact object $F$ of $\QCoh(X)$ with support $|X|-\calU$.
\end{enumerate}
In the special case of ordinary quasi-compact quasi-separated algebraic stacks, the $\beta$-Thomason condition is equivalent to the following condition: for each quasi-compact open immersion $U\rightarrow X$, the full subcategory $\QCoh_{X-U}(X)\subseteq \QCoh(X)$ spanned by those objects which are supported on $|X|-|U|$ (see \ref{the full subcategory spanned by objects with prescribed support}) is compactly generated by a set of cardinality at most $\beta$. Although its spectral analogue is not evident (the author does not know if \cite[4.10]{MR3705292} is true in the spectral setting), we will see in \ref{twisted compact generation and global sections} and \ref{compact objects with prescribed support} that if a quasi-geometric spectral algebraic stack $X$ is of twisted compact generation, then it not only satisfies the $\beta$-Thomason condition for some $\beta$, but also the aforementioned condition on $\QCoh_{X-U}(X)$. 
\end{remark}

\begin{pg}
	The main ingredient in our proof of \ref{Brauer spaces and Azumaya algebras for quasi-geometric spectral algebraic stacks} is the following result, which is closely related to \cite[4.8]{MR2957304} (which asserts that for each quasi-compact quasi-separated derived scheme $X$ and each element $\alpha \in \HH^2_{\et}(X, \mathbb{G}_m)$, the $\alpha$-twisted derived dg-category $\mathrm{L}_\alpha(X)$ admits a compact generator):
\end{pg}

\begin{theorem}\label{twisted compact generation of spectral algebraic stacks}
	Let $X$ be a quasi-geometric spectral algebraic stack which admits a quasi-finite presentation. Then $X$ is of twisted compact generation. In particular, for each quasi-compact open subset $\calU\subseteq |X|$, there exists a compact object $F$ of $\QCoh(X)$ with support $|X|-\calU$. 
\end{theorem}

\begin{remark}\noindent
\begin{enumerate}[(i)]
\item According to \cite[10.3.2.1]{lurie2018sag}, if $\sfX$ is a quasi-compact quasi-separated spectral algebraic space and $\calC$ is a compactly generated prestable quasi-coherent stack on $\sfX$, then the $\infty$-category $\QCoh(\sfX; \calC)$ is compactly generated. In the case where $\calC$ is stable, we can deduce the statement from \ref{twisted compact generation of spectral algebraic stacks}.

\item Let $X$ be an ordinary quasi-compact algebraic stack with quasi-finite and separated diagonal. Then \cite[Theorem A]{MR3705292} shows that for each quasi-compact open subset $\calU\subseteq |X|$, there exists a compact complex $F \in \mathrm{D}_{\mathrm{qc}}(X)$ with support $|X|-\calU$ (cf. \ref{unbounded derived category}). This can be viewed as a special case of \ref{twisted compact generation of spectral algebraic stacks}. 

\item If $\sfX$ is a quasi-compact quasi-separated spectral algebraic space and $\calU\subseteq |\sfX|$ is a quasi-compact open subset, then \cite[11.1.2.1]{lurie2018sag} guarantees that there exists a perfect object $F \in \QCoh(X)$ with support $|\sfX|-\calU$. This can be obtained from \ref{twisted compact generation of spectral algebraic stacks}. 
\end{enumerate}
\end{remark}

\begin{pg}
	The main difficulty in the proof of \ref{twisted compact generation of spectral algebraic stacks} is that we do not know if the compact generation of $R$-linear $\infty$-categories (see \cite[D.1.2.1]{lurie2018sag}), where $R$ is an $\bbE_\infty$-ring, is local for the fpqc topology. On the other hand, in the setting of derived algebraic geometry, \cite[4.13]{MR2957304} (which asserts that the existence of a compact generator of a locally presentable dg-category over a simplicial commutative ring is local for the fppf topology) is essential to the proof of \cite[4.8]{MR2957304}. As Antieau--Gepner mentioned in \cite[p.1215]{MR3190610}, To\"en's proof of \cite[4.13]{MR2957304} makes use of quotients of simplicial commutative rings, and therefore it cannot be carried out in the spectral setting due to the lack of quotient construction of $\bbE_\infty$-rings. Nonetheless, Antieau--Gepner showed that the existence of a compact generator of an $R$-linear $\infty$-category, where $R$ is an $\bbE_\infty$-ring, is local for the \'etale topology (see \cite[6.16]{MR3190610}), and attributed the idea of the proof to Lurie of \cite[6.1]{lurie2011dagxi} in which the key ingredient is that the compact generation of $R$-linear $\infty$-category satisfies descent for the Nisnevich topology. Moreover, the notion of \emph{scallop decomposition} of \cite[2.5.3.1]{lurie2018sag}, which is closely related to the Nisnevich topology, plays a crucial role in the proof of \cite[10.3.2.1]{lurie2018sag}: a scallop decomposition of a spectral Deligne-Mumford stack $\sfX$ consists of a sequence of open immersions $\emptyset \simeq \sfU_0\rightarrow \sfU_1 \rightarrow \cdots \rightarrow \sfU_n \simeq \sfX$ such that for every $1\leq i \leq n$, there exists an excision square of spectral Deligne-Mumford stacks
$$
\Pull{\sfV}{\sfY}{\sfU_{i-1}}{\sfU_i}{}{}{}{}
$$
where $\sfY$ is affine and $\sfV$ is quasi-compact. However, the essential difficulty in extending \cite[10.3.2.1]{lurie2018sag} from quasi-compact quasi-separated spectral algebraic spaces to quasi-geometric spectral algebraic stacks is that the notion of scallop decomposition is designed to accommodate spectral algebraic spaces: more concretely, if a spectral Deligne-Mumford admits a scallop decomposition, then it must be a quasi-compact quasi-separated spectral algebraic space (see \cite[3.4.2.1]{lurie2018sag}). To address this difficulty, we develop a theory of the underlying topological space of quasi-geometric spectral algebraic stacks (see \ref{the induced topological space of quasi-geometric stacks}), so that we can extend the definitions of an excision square and a scallop decomposition to those stacks, without imposing the requirement that $\sfY$ appearing in the excision square above is affine (see \ref{excision squares of quasi-geometric stacks} and \ref{stacky scallop decompositions}). We will then prove \ref{twisted compact generation of spectral algebraic stacks} by generalizing the ``induction principle" for ordinary algebraic stacks by Hall--Rydh \cite[Theorem E]{MR3754421} to quasi-geometric spectral algebraic stacks. More precisely, we show the representability of the spectral Hilbert functor to provide a special presentation of quasi-geometric spectral algebraic stacks (which is an analogue of \cite[4.1]{MR3754421} in the spectral setting; see \ref{a presentation for applying quasi-finite devissage}), from which we are reduced to proving that the property of being of twisted compact generation satisfies descent for finite morphisms and excision squares of quasi-geometric spectral algebraic stacks (see \ref{finite morphisms and twisted compact generation} and \ref{excision squares and twisted compact generation}).
\end{pg}

\begin{remark}
	As a consequence of our proof of \ref{twisted compact generation of spectral algebraic stacks}, we will see in \ref{compact generation is local for the quasi-finite topology} that for stable $R$-linear $\infty$-categories, where $R$ is a connective $\bbE_\infty$-ring, the property of being compactly generated satisfies descent with respect to the maps which are quasi-finite, faithfully flat, and almost of finite presentation. This is a generalization of \cite[D.5.3.1]{lurie2018sag} in the stable case, which asserts that the property is local for the \'etale topology.
\end{remark}

\subsection{Outline of the paper}
	In Section \ref{Sec: Spectral Algebraic Stacks}, we formulate quasi-compact algebraic stacks with quasi-affine diagonal in the setting of derived and spectral algebraic geometry, and study the relationship between them. In Section \ref{Sec: Points of Quasi-geometric Spectral Algebraic Stacks}, we define the underlying topological space of quasi-geometric stacks and establish some of its basic properties. In Section \ref{Sec: Excision Squares}, we first introduce excision squares, stacky scallop decompositions, and Nisnevich coverings of quasi-geometric spectral algebraic stacks. We then show the representability of the Hilbert functor in the spectral setting to provide a special presentation of those stacks. Section \ref{Sec: Twisted Compact Generations} is devoted to introducing the notion of twisted compact generation for quasi-geometric spectral algebraic stacks and developing some descent results. In Section \ref{Sec: Brauer Spaces and Azumaya Algebras}, we study the extended Brauer groups and Azumaya algebras in spectral algebraic geometry. 

\subsection{Conventions}
	We will follow the set-theoretic convention of \cite{MR2522659}. 
	
\subsection{Acknowledgements} 
	The author is grateful to Benjamin Antieau, Jack Hall, and Bertrand To\"en for helpful comments and conversations. This work was supported by IBS-R003-D1.

\section{Spectral Algebraic Stacks}\label{Sec: Spectral Algebraic Stacks}
	In this section, we introduce the basic objects of study in this paper: quasi-geometric spectral algebraic stacks. We also investigate the relationship between derived algebraic geometry and spectral algebraic geometry, so that we can incorporate quasi-geometric derived algebraic stacks (which are defined similarly) into our study.
	
\begin{pg}
	For an $\bbE_\infty$-ring $R$, let $\CAlg_R$ denote the $\infty$-category $\bbE_\infty$-algebras over $R$ (see \cite[7.1.3.8]{lurie2017ha}). The following big \'etale topology on the opposite $\infty$-category $\CAlg_R^{\op}$ (whose existence is evident from the small \'etale topology \cite[B.6.2.1]{lurie2018sag}) is ubiquitous in \cite{lurie2018sag}, but not mentioned explicitly. We record it here for reference:
\end{pg}

\begin{lemma}\label{big etale topology}
	Let $R$ be an $\bbE_\infty$-ring. Then there exits a Grothendieck topology on the $\infty$-category $\CAlg_R^{\op}$ which can be characterized as follows: if $A$ is an $\bbE_\infty$-algebra over $R$, then a sieve $\calC \subset (\CAlg_R^{\op})_{/A} \simeq \CAlg_A^{\op}$ is a covering if and only if it contains a finite collection of maps $\{A\rightarrow A_i\}_{1\leq i \leq n}$ for which the induced map $A\rightarrow \prod_{1\leq i \leq n}A_i$ is faithfully flat and \'etale.
\end{lemma}

\begin{remark}\label{the infty-category of etale sheaves over R}
	In the case where $R$ is connective, the same proof provides an apparent analogue for the $\infty$-category $\CAlg^{\cn}_R$ of connective $\bbE_\infty$-algebras over $R$. Let $\widehat{\Shv}_{\et}(\CAlg^{\cn}_R) \subseteq \Fun(\CAlg^{\cn}_R, \widehat{\SSet})$ denote the full subcategory spanned by the \'etale sheaves (here $\widehat{\SSet}$ denotes the $\infty$-category of (not necessarily small) spaces; see \cite[1.2.16.4]{MR2522659}).
\end{remark}

\begin{pg}\label{properties of representable morphisms}
	Let $R$ be	a connective $\bbE_\infty$-ring. We say that a morphism $X\rightarrow Y$ in $\Fun(\CAlg^{\cn}_R, \widehat{\SSet})$ is \emph{representable} if, for every connective $R$-algebra $R'$ and every morphism $\Spec R' \rightarrow Y$, the fiber product $X\times_Y \Spec R'$ is representable by a spectral Deligne-Mumford stack over $\Spec R$ (cf. \cite[6.3.2.1]{lurie2018sag}). Let $P$ be a property of morphism of spectral Deligne-Mumford stacks which is local on the target with respect to the \'etale topology \cite[6.3.1.1]{lurie2018sag} and stable under base change \cite[6.3.3.1]{lurie2018sag}. We say that a representable morphism $X\rightarrow Y$ \emph{has the property P} if, for every connective $\bbE_\infty$-algebra $R'$ over $R$ and every morphism $\Spec R' \rightarrow Y$, the projection $X\times_Y \Spec R' \rightarrow \Spec R'$, which can be identified with a morphism of spectral Deligne-Mumford stacks, has the property $P$ (cf. \cite[6.3.3.3]{lurie2018sag}).
\end{pg}

\begin{pg}
	According to \cite[9.1.0.1]{lurie2018sag}, a \emph{quasi-geometric stack} is a functor $X:\CAlg^{\cn} \rightarrow \widehat{\SSet}$ satisfying the following conditions:
\begin{enumerate}[(i)]
\item The functor $X$ is a sheaf for the fpqc topology of \cite[B.6.1.3]{lurie2018sag}.
\item The diagonal $\Delta: X \rightarrow X \times X$ is representable and quasi-affine (see \cite[6.3.3.6]{lurie2018sag}).
\item There exists a faithfully flat morphism $\Spec A \rightarrow X$, where $A$ is a connective $\bbE_\infty$-ring; see \cite[6.3.3.7]{lurie2018sag}.
\end{enumerate}
We now introduce a special class of quasi-geometric stacks, called \emph{quasi-geometric spectral algebraic stacks}. This collection of quasi-geometric stacks admits a ``smooth covering". There are at least two different ways to construct a suitable ``smoothness" in the setting of spectral algebraic geometry: for example, fiber smoothness and differentially smoothness (see \cite[11.2.5.5]{lurie2018sag}). Fiber smooth morphisms are closely related to smooth morphisms in the classical algebraic geometry (see \ref{characterization of fiber smooth morphism}), and we adopt the fiber smoothness in our definition of quasi-geometric spectral algebraic stacks (the terminology is not standard):
\end{pg}

\begin{definition}\label{quasi-geometric spectral algebraic stack}
	Let $R$ be	a connective $\bbE_\infty$-ring. A \emph{quasi-geometric spectral algebraic stack over $R$} is a functor $X:\CAlg^{\cn}_R \rightarrow \widehat{\SSet}$ which satisfies the following conditions:
\begin{enumerate}[(i)]
\item\label{fpqc sheaf} The functor $X$ is a sheaf for the fpqc topology.
\item\label{quasi-affine diagonal} The diagonal morphism $\Delta: X \rightarrow X \times X$ is representable and quasi-affine.
\item\label{smooth cover} There exist a connective $\bbE_\infty$-algebra $A$ over $R$ and a morphism $\Spec A \rightarrow X$ which is fiber smooth and surjective. 
\end{enumerate}
In the special case where $R$ is the sphere spectrum (that is, an initial object of $\CAlg^{\cn}$), we simply say that $X$ is a \emph{quasi-geometric spectral algebraic stack}. In other words, a quasi-geometric spectral algebraic stack is a quasi-geometric stack which satisfies condition (\ref{smooth cover}). 
\end{definition}

\begin{remark}\label{quasi-geometric spectral algebraic stacks over any base} 
	Let $\Fun(\CAlg^{\cn}, \widehat{\SSet})_{/R}$ denote the slice $\infty$-category $\Fun(\CAlg^{\cn}, \widehat{\SSet})_{/\Spec R}$. Let $X':\CAlg^{\cn}_R\rightarrow \widehat{\SSet}$ be the image of an object $(X\rightarrow \Spec R) \in \Fun(\CAlg^{\cn}, \widehat{\SSet})_{/R}$ under the equivalence of $\infty$-categories $\Fun(\CAlg^{\cn}, \widehat{\SSet})_{/R} \simeq \Fun(\CAlg^{\cn}_R, \widehat{\SSet})$. Then $X'$ is a quasi-geometric spectral algebraic stack over $R$ if and only if the functor $X$ is a quasi-geometric spectral algebraic stack (over the sphere spectrum). 
\end{remark}

\begin{example}
	Quasi-geometric spectral algebraic stacks exist in abundance: 
\begin{enumerate}[(i)]
\item Every quasi-geometric spectral Deligne-Mumford stack is a quasi-geometric spectral algebraic stack because every \'etale morphism is fiber smooth \cite[11.2.3.2]{lurie2018sag}. 
\item Let $R$ be a commutative ring. We will see in \ref{quasi-geometric spectral algebraic stacks from DAG} that for each quasi-geometric derived algebraic stack over $R$, there is an underlying quasi-geometric spectral algebraic stack over $R$; in particular, each ordinary quasi-compact algebraic stack over $R$ with quasi-affine diagonal can be regarded as a quasi-geometric spectral algebraic stack over $R$.
\end{enumerate}
\end{example}

\begin{pg}
	Our choice of ``smooth coverings" in the definition of quasi-geometric spectral algebraic stacks \ref{quasi-geometric spectral algebraic stack} is motivated by the following characterization of fiber smoothness in terms of the $0$-truncations of \cite[1.4.6.5]{lurie2018sag}:
\end{pg}

\begin{lemma}\label{characterization of fiber smooth morphism}
	Let $f:\sfX\rightarrow \sfY$ be a morphism of spectral algebraic spaces. Then $f$ is fiber smooth if and only if it is flat and the underlying morphism of ordinary algebraic spaces $\tau_{\leq 0}(f)$ is smooth.
\end{lemma}
\begin{proof}
	The assertion is \'etale-local on $\sfX$ and $\sfY$, so we may assume that $\sfX$ and $\sfY$ are affine. In this case, the desired result is an immediate consequence of \cite[11.2.3.5]{lurie2018sag} and \cite[11.2.4.1]{lurie2018sag}. 
\end{proof}

\begin{pg} 
	In classical algebraic geometry, the big smooth topology does not play as significant role as the big \'etale topology. This is in part due to the fact that a smooth surjection of ordinary schemes \'etale-locally admits a section (see \cite[17.16.3]{MR0238860}), and therefore the topoi induced by these topologies are equivalent. In the spectral setting, an analogous statement holds for differentially smooth morphisms by virtue of \cite[4.47]{MR3190610}. The following topology defined by fiber smooth maps (whose existence can be proven in the same way as \cite[B.6.1.3]{lurie2018sag}) is nevertheless of interest to us in this paper:
\end{pg}

\begin{lemma}\label{fiber smooth topology}
	Let $R$ be a connective $\bbE_\infty$-ring. Then there exists a Grothendieck topology on the $\infty$-category $(\CAlg^{\cn}_R)^{\op}$ which can be described as follows: if $A$ is a connective $\bbE_\infty$-algebra over $R$, then a sieve $\calC \subset (\CAlg_R^{\op})_{/A} \simeq \CAlg_A^{\op}$ is a covering if and only if it contains a finite collection of maps $\{A\rightarrow A_i\}_{1\leq i \leq n}$ for which the induced map $A\rightarrow \prod_{1\leq i \leq n}A_i$ is faithfully flat and fiber smooth.
\end{lemma}

\begin{remark}\label{the infty-category of fiber smooth sheaves over R}
	We refer to the Grothendieck topology of \ref{fiber smooth topology} as the \emph{fiber smooth topology on $(\CAlg^{\cn}_R)^{\op}$}. We will see later in \ref{the necessity of fiber smooth topology} that it has the virtue of connecting quasi-coherent stacks in derived algebraic geometry and spectral algebraic geometry. Let $\widehat{\Shv}_{\fsm}(\CAlg^{\cn}_R) \subseteq \Fun(\CAlg^{\cn}_R, \widehat{\SSet})$ denote the full subcategory spanned by the fiber smooth sheaves.
\end{remark}

\begin{remark}\label{fiber smooth sheaf condition for quasi-geometric spectral algebraic stacks}
	In the situation of \ref{quasi-geometric spectral algebraic stack}, we can replace (\ref{fpqc sheaf}) by the apparently weaker condition that $X$ is a sheaf for the fiber smooth topology. Indeed, if $X$ satisfies this condition along with conditions (\ref{quasi-affine diagonal}) and (\ref{smooth cover}) of \ref{quasi-geometric spectral algebraic stack}, then $X$ is a (hypercomplete) sheaf with respect to the fpqc topology; this can be established by mimicking the proof of \cite[9.1.4.3]{lurie2018sag}.
\end{remark}

\begin{pg}
	For the rest of this section, we study how to deal with derived stacks in the context of spectral algebraic geometry. In particular, we will see in \ref{quasi-geometric spectral algebraic stacks from DAG} that one can associate a quasi-geometric spectral algebraic stack to each quasi-geometric derived algebraic stack. This connection has the virtue of allowing us to apply our main theorems \ref{Brauer spaces and Azumaya algebras for quasi-geometric spectral algebraic stacks} and \ref{twisted compact generation of spectral algebraic stacks}---which are described in terms of quasi-geometric spectral algebraic stacks---to quasi-geometric derived algebraic stacks as well.
\end{pg}

\begin{pg}
	Let $R$ be a commutative ring. Let $\CAlg^\Delta_R$ denote the $\infty$-category of simplicial commutative $R$-algebras (see, for example, \cite[25.1.1.1]{lurie2018sag}). It follows from \cite[25.1.2.1]{lurie2018sag} that there is a forgetful functor 
$$
\Theta_R:\CAlg^\Delta_R \rightarrow \CAlg^{\cn}_R.
$$
We denote the image of $A\in \CAlg^\Delta_R$ under $\Theta_R$ by $A^\circ$ and refer to it as the \emph{underlying $\bbE_\infty$-algebra} of $A$. By virtue of \cite[4.3.3.7]{MR2522659}, the restriction functor $\Theta_R^\ast: \Fun(\CAlg^{\cn}_R, \widehat{\SSet}) \rightarrow \Fun(\CAlg^\Delta_R, \widehat{\SSet})$ admits a left adjoint ${\Theta_R}_!$ which carries each functor $X:\CAlg^\Delta_R\rightarrow \widehat{\SSet}$ to its left Kan extension along $\Theta_R$. 
\end{pg}

\begin{remark}\label{big etale topology for simplicial commutative rings}
	There are evident analogues of the \'etale and fpqc topologies (see \ref{big etale topology} and \cite[B.6.1.3]{lurie2018sag}) for the $\infty$-category $(\CAlg^\Delta_R)^{\op}$. 
\end{remark}

\begin{pg}
	Let $\mathring{\Theta}:\CAlg^\Delta\rightarrow \CAlg^{\cn}$ denote the composition of $\Theta_\mathbb{Z}$ with the forgetful functor $\CAlg^{\cn}_\mathbb{Z}\rightarrow \CAlg^{\cn}$. To every derived Deligne-Mumford stack $\sfX=(\calX, \calO_{\calX})$ (which can be defined as in \cite[1.4.4.2]{lurie2018sag}, using $\CAlg^\Delta$ in place of $\CAlg^{\cn}$), one can associate a spectral Deligne-Mumford stack $(\calX, \mathring{\Theta} \circ \calO_{\calX})$, which we denote by $\sfX^\circ$ and refer to as the \emph{underlying spectral Deligne-Mumford stack} of $\sfX$. We can regard this construction as a functor from the $\infty$-category $\DerDM$ of derived Deligne-Mumford stacks to the $\infty$-category $\SpDM$ of spectral Deligne-Mumford stacks; it carries the affine spectrum of a simplicial commutative ring $A$ to the affine spectrum of the underlying $\bbE_\infty$-ring $A^\circ$. Let $\DerDM_{/R}$ and $\SpDM_{/R}$ denote the slice $\infty$-categories $\DerDM_{/\Spec R}$ and $\SpDM_{/\Spec R}$, respectively. We then have a functor 
$$
\DerDM_{/R} \rightarrow \SpDM_{/R}
$$
which carries a derived Deligne-Mumford stack $\sfX$ over $R$ to its underlying spectral Deligne-Mumford stack $\sfX^\circ$ over $R$. Let $L_{\et}:\Fun(\CAlg^{\cn}_R,\widehat{\SSet}) \rightarrow \widehat{\Shv}_{\et}(\CAlg^{\cn}_R)$ denote a left adjoint to the inclusion (see \ref{the infty-category of etale sheaves over R}). To extend the construction $\sfX\mapsto \sfX^\circ$ to derived (algebraic) stacks, we need the following ``functor of points" perspective (cf. \cite[9.27]{lurie2011dagvii}):
\end{pg}

\begin{lemma}\label{from DAG to SAG via functors} 
	Let $R$ be a commutative ring. Then the composite functor
$$
L_{\et}\circ {\Theta_R}_!:\Fun(\CAlg^\Delta_R, \widehat{\SSet}) \rightarrow \Fun(\CAlg^{\cn}_R, \widehat{\SSet}) \rightarrow \widehat{\Shv}_{\et}(\CAlg^{\cn}_R)
$$
restricts to the functor $\DerDM_{/R} \rightarrow \SpDM_{/R}$.  
\end{lemma}
\begin{proof}
	Suppose we are given a derived Deligne-Mumford stack $\sfX=(\calX, \calO_{\calX})$ over $R$. Let $h_{\sfX}:\CAlg^\Delta_R \rightarrow \widehat{\SSet}$ denote the functor represented by $\sfX$ (given by the formula $h_{\sfX}(A)=\Map_{\DerDM_{/R}}(\Spec A, \sfX)$), and define $h_{\sfX^\circ}$ similarly. We wish to show that the natural morphism of functors
$$
(\ast_X): L_{\et}({\Theta_R}_! h_{\sfX}) \rightarrow h_{\sfX^\circ}
$$
is an equivalence (note that $h_{\sfX^\circ}$ is a sheaf for the \'etale topology). Let $\calX_0$ be the full subcategory of $\calX$ spanned by those objects $U\in \calX$ for which $(\ast_{X_U})$ is an equivalence, where $X_U:\CAlg^\Delta_R \rightarrow \widehat{\SSet}$ denotes the functor represented by the derived Deligne-Mumford stack $\sfX_U=(\calX_{/U}, {\calO_{\calX}} |_U)$ over $R$. It follows immediately that $\calX_0$ contains all affine objects $U \in \calX$. By virtue of \cite[1.4.7.9]{lurie2018sag}, it will suffice to show that $\calX_0$ is closed under small colimits in $\calX$. To prove this, suppose we are given a small diagram $\{U_\alpha\}$ in $\calX_0$ having a colimit $U\in \calX$. We then have a commutative diagram in $\widehat{\Shv}_{\et}(\CAlg^{\cn}_R)$
$$
\Pull{\colim L_{\et}{\Theta_R}_! h_{\sfX_{U_\alpha}}}{\colim h_{\sfX_{U_\alpha}^\circ}}{L_{\et}{\Theta_R}_! h_{\sfX_U}}{h_{\sfX_U^\circ}.}{}{}{}{}
$$
Since the transition morphisms in the diagram $\{\sfX_{U_\alpha}^\circ\}$ are \'etale, the right vertical arrow is an equivalence. Combing the analogous equivalence for the diagram $\{\sfX_{U_\alpha}\}$ of derived Deligne-Mumford stacks with the fact that the composition $L_{\et}\circ {\Theta_R}_!$ commutes with small colimits, we see that the left vertical arrow is also an equivalence, thereby completing the proof.
\end{proof}

\begin{remark}\label{from DAG to SAG via functors in the case of quasi-geometric stacks} 
	Let us say that a derived Deligne-Mumford stack $\sfX$ is \emph{quasi-geometric} if it is quasi-compact and the diagonal $\Delta_{\sfX}:\sfX \rightarrow \sfX\times \sfX$ is quasi-affine. In this case, the underlying spectral Deligne-Mumford stack $\sfX^\circ$ is quasi-geometric (see \cite[9.1.4.1]{lurie2018sag}), so that the functor $h_{\sfX^\circ}$ that it represents is a (hypercomplete) sheaf with respect to the fpqc topology by virtue of \cite[9.1.4.3]{lurie2018sag}; in particular, it satisfies descent for the fiber smooth topology. 
	
	Let $L_{\fsm}: \Fun(\CAlg^{\cn}_R, \widehat{\SSet}) \rightarrow \widehat{\Shv}_{\fsm}(\CAlg^{\cn}_R)$ denote a left adjoint to the inclusion functor (see \ref{the infty-category of fiber smooth sheaves over R}). Arguing as in the proof of \ref{from DAG to SAG via functors} (using $L_{\fsm}h_{\sfX_{U_\alpha}^\circ}$ and $L_{\fsm}h_{\sfX_U^\circ}$ in place of $h_{\sfX_{U_\alpha}^\circ}$ and $h_{\sfX_U^\circ}$, respectively), we deduce that the composite functor
$$
L_{\fsm}\circ {\Theta_R}_!:\Fun(\CAlg^\Delta_R, \widehat{\SSet}) \rightarrow \Fun(\CAlg^{\cn}_R, \widehat{\SSet}) \rightarrow \widehat{\Shv}_{\fsm}(\CAlg^{\cn}_R)
$$
carries a quasi-geometric derived Deligne-Mumford stack $\sfX$ over $R$ to its underlying quasi-geometric spectral Deligne-Mumford stack $\sfX^\circ$ over $R$.
\end{remark}

\begin{pg}
	According to \cite[3.4.7]{MR2717174}, a morphism $A\rightarrow B$ in $\CAlg^\Delta$ is \emph{smooth} if the relative algebraic cotangent complex $L^{\mathrm{alg}}_{B/A}$ (see \cite[3.2.14]{MR2717174} and \cite[25.3.2.1]{lurie2018sag}) is a dual of connective and perfect object of $\Mod_B$ and almost of finite presentation (see \cite[3.1.5]{MR2717174}). The following observation shows a close relationship between smooth maps in the derived setting and fiber smooth maps in the spectral setting:
\end{pg}	

\begin{lemma}\label{smoothness in DAG and SAG}
	Let $f:A\rightarrow B$ be a morphism of simplicial commutative rings. Then $f$ is smooth if and only if the underlying morphism of $\bbE_\infty$-rings $f^\circ:A^\circ \rightarrow B^\circ$ is fiber smooth. 
\end{lemma}
\begin{proof}
	By virtue of \cite[3.4.9]{MR2717174}, $f$ is smooth if and only if it is flat and $\pi_0(f)$ is a smooth map of commutative rings, so the desired result follows from \ref{characterization of fiber smooth morphism}.
\end{proof}	

\begin{remark}\label{big smooth topology for simplicial commutative rings}
	There is an analogue of the fiber smooth topology of \ref{fiber smooth topology} for the $\infty$-category $(\CAlg^\Delta_R)^{\op}$ by replacing ``fiber smooth" with ``smooth" in \ref{fiber smooth topology}. We refer to this Grothendieck topology as the \emph{smooth topology on $(\CAlg^\Delta_R)^{\op}$}. Let $\widehat{\Shv}_{\sm}(\CAlg^\Delta_R) \subseteq \Fun(\CAlg^{\Delta}_R, \widehat{\SSet})$ denote the full subcategory spanned by the smooth sheaves.
\end{remark}

\begin{pg}	
	Using an evident analogue of \ref{properties of representable morphisms}, we introduce a variant of \ref{quasi-geometric spectral algebraic stack} in the derived setting:
\end{pg}

\begin{definition}\label{quasi-geometric derived algebraic stack}
	Let $R$ be a commutative ring. A \emph{quasi-geometric derived algebraic stack over $R$} is a functor $X:\CAlg^\Delta_R \rightarrow \widehat{\SSet}$ which satisfies the following conditions:
\begin{enumerate}[(i)]
\item\label{fpqc sheaf in DAG} The functor $X$ is a sheaf for the fpqc topology of \ref{big etale topology for simplicial commutative rings}.
\item\label{quasi-affine diagonal in DAG} The diagonal morphism $\Delta: X \rightarrow X \times X$ is representable and quasi-affine.
\item\label{smooth cover in DAG} There exists a simplicial commutative algebra $A$ over $R$ and a morphism $\Spec A \rightarrow X$ which is smooth and surjective. 
\end{enumerate}
\end{definition}

\begin{remark}
	Arguing as in the proof of \cite[4.47]{MR3190610}, we see that a smooth morphism of derived Deligne-Mumford stacks \'etale-locally admits a section. Using this observation, a minor modification of the proof of \cite[9.1.4.3]{lurie2018sag} guarantees that an \'etale sheaf satisfying conditions (\ref{quasi-affine diagonal in DAG}) and (\ref{smooth cover in DAG}) of \ref{quasi-geometric derived algebraic stack} is a sheaf for the fpqc topology. In the situation of \ref{quasi-geometric derived algebraic stack}, one can therefore replace (\ref{fpqc sheaf in DAG}) by the weaker condition that $X$ is a sheaf for the \'etale topology. 
\end{remark}

\begin{pg}
	In order to extend \ref{from DAG to SAG via functors} from derived Deligne-Mumford stacks to (quasi-geometric) derived algebraic stacks, we need to understand if the functor $L_{\et}\circ {\Theta_R}_!$ carries representable morphisms in $\Fun(\CAlg^\Delta_R, \widehat{\SSet})$ to representable morphisms in $\widehat{\Shv}_{\et}(\CAlg^{\cn}_R)$; however, this is not straightforward at all. We will circumvent this difficulty by using the following lemmas:
\end{pg}

\begin{lemma}\label{a smooth surjection in DAG admits a section etale locally}
	Let $f: X\rightarrow Y$ be a representable morphism in the $\infty$-category of \'etale sheaves on $(\CAlg^\Delta_R)^{\op}$. If $f$ is a smooth surjection, then $f$ is an effective epimorphism and $(L_{\fsm} \circ {\Theta_R}_!)(f)$ is an effective epimorphism of fiber smooth sheaves on $(\CAlg^{\cn}_R)^{\op}$.
\end{lemma}
\begin{proof}
	The forgetful functor $\Theta_R:\CAlg^\Delta_R \rightarrow \CAlg^{\cn}_R$ is left exact (see \cite[25.1.2.2]{lurie2018sag}) and carries smooth coverings to fiber smooth coverings (see \ref{smoothness in DAG and SAG}). In particular, the restriction functor $\Theta_R^\ast$ restricts to a morphism of $\infty$-topoi $\widehat{\Shv}_{\fsm}(\CAlg^{\cn}_R) \rightarrow \widehat{\Shv}_{\sm}(\CAlg^\Delta_R)$ (see \ref{big smooth topology for simplicial commutative rings}), whose left adjoint is given by the composition of the inclusion $\widehat{\Shv}_{\sm}(\CAlg^\Delta_R)  \subseteq \Fun(\CAlg^{\Delta}_R, \widehat{\SSet})$ with $L_{\fsm} \circ {\Theta_R}_!$. Since a left adjoint of a geometric morphism of $\infty$-topoi preserves effective epimorphisms \cite[6.2.3.6]{MR2522659}, it will suffice to show that $f$ is an effective epimorphism in $\widehat{\Shv}_{\sm}(\CAlg^\Delta_R)$. By virtue of \cite[3.4.4]{MR2717174}, a smooth surjective morphism $f:\sfX\rightarrow \sfY$ of derived Deligne-Mumford stacks satisfies an infinitesimal lifting criterion. Using the argument of \cite[4.47]{MR3190610}, we deduce that $f$ is an effective epimorphism of \'etale sheaves, which implies the desired result. 
\end{proof}

\begin{pg}
	We are now ready to prove the main result of this section:
\end{pg}

\begin{proposition}\label{quasi-geometric spectral algebraic stacks from DAG}
	Let $R$ be a commutative ring. Let $X$ be a quasi-geometric derived algebraic stack over $R$. Then the functor $(L_{\fsm}\circ {\Theta_R}_!)(X): \CAlg^{\cn}_R\rightarrow \widehat{\SSet}$ is a quasi-geometric spectral algebraic stack over $R$. 
\end{proposition}
\begin{proof}
	Choose a smooth surjection $p:\Spec A \rightarrow X$ where $A$ is a simplicial commutative $R$-algebra. Let $\sfY$ be a derived Deligne-Mumford stack representing the fiber product $\Spec A \times_X \Spec A$ (note that $\sfY$ is quasi-affine), so that it fits into a pullback square of fpqc sheaves
$$
\Pull{\sfY}{\Spec A}{\Spec A}{X.}{q}{}{p}{p}
$$
Using \ref{from DAG to SAG via functors in the case of quasi-geometric stacks} and the fact that the functor $L_{\fsm}\circ {\Theta_R}_!$ is left exact (see the proof of \ref{a smooth surjection in DAG admits a section etale locally}), the above diagram induces a pullback square of fiber smooth sheaves on $(\CAlg^{\cn}_R)^{\op}$
$$
\Pull{\sfY^\circ}{\Spec A^\circ}{\Spec A^\circ}{L_{\fsm}({\Theta_R}_!(X)).}{q^\circ}{}{p^\circ}{p^\circ}
$$
Since $q:\sfY\rightarrow \Spec A$ is quasi-affine, so is its underlying morphism $q^\circ$ of spectral Deligne-Mumford stacks. By virtue of \ref{a smooth surjection in DAG admits a section etale locally}, the map $p^\circ: \Spec A^\circ\rightarrow L_{\fsm}({\Theta_R}_!(X))$ is an effective epimorphism, so that a relative version of \cite[9.1.1.3]{lurie2018sag} with the fiber smooth topology in place of the fpqc topology (which can be proven by exactly the same argument) guarantees that $p^\circ$ is representable quasi-affine. Since $q$ is a smooth surjection, it follows from \ref{smoothness in DAG and SAG} that $q^\circ$ is fiber smooth and surjective. Consequently, the representable morphism $p^\circ$ is also a fiber smooth surjection because the property of being a fiber smooth morphism is local on the target with respect to the flat topology (see \cite[11.2.5.9]{lurie2018sag}). Applying a variant of \cite[9.1.1.2]{lurie2018sag} which uses fiber smooth morphisms in place of flat morphisms to $L_{\fsm}({\Theta_R}_!(X))$ and $p^\circ$, we deduce that the diagonal of $L_{\fsm}({\Theta_R}_!(X))$ is representable quasi-affine. Invoking \ref{fiber smooth sheaf condition for quasi-geometric spectral algebraic stacks}, we conclude that $L_{\fsm}({\Theta_R}_!(X))$ is a quasi-geometric spectral algebraic stack over $R$. 
\end{proof}

\begin{remark}\label{underlying quasi-geometric spectral algebraic stacks}
	Let $X$ be a quasi-geometric derived algebraic stack $X$ over $R$. Let $X^\circ$ denote its image under the functor $L_{\fsm}\circ {\Theta_R}_!$; we refer to $X^\circ$ as the \emph{underlying quasi-geometric spectral algebraic stack over $R$} of $X$.
\end{remark}

\section{Points of Quasi-geometric Spectral Algebraic Stacks}\label{Sec: Points of Quasi-geometric Spectral Algebraic Stacks}
	In this section, we define the notion of points of quasi-geometric spectral algebraic stacks in such a way that the points of ordinary algebraic stacks are defined (see \cite[5.2]{MR1771927}) and establish some of their basic properties. 

\begin{definition}\label{the induced topological space of quasi-geometric stacks}
	Let $X:\CAlg^{\cn}\rightarrow \widehat{\SSet}$ be a functor satisfying the following condition:
\begin{itemize}
\item[$(\ast)$] There exist a quasi-separated spectral algebraic space $\sfX_0$ and a relative spectral algebraic space $\pi:\sfX_0 \rightarrow X$ in $\Fun(\CAlg^{\cn},\widehat{\SSet})$ which is quasi-separated, faithfully flat, and locally almost of finite presentation.
\end{itemize}
A \emph{point} of $X$ is a morphism $\Spec \kappa \rightarrow X$, where $\kappa$ is a field. We define an equivalence relation on the (not necessarily small) set of points of $X$ as follows: given two points $p:\Spec \kappa \rightarrow X$ and $p':\Spec \kappa' \rightarrow X$, we will write $p\sim p'$ if there exists a field $\kappa''$ and a commutative diagram
$$
\Pull{\Spec \kappa''}{\Spec \kappa'}{\Spec \kappa}{X}{}{}{p'}{p}
$$
in $\Fun(\CAlg^{\cn},\widehat{\SSet})$; let $|X|$ denote the set of equivalence classes. We endow $|X|$ with the topology generated by the sets $|U|$, where $U$ ranges over all representable open $j:U\rightarrow X$ in $\Fun(\CAlg^{\cn},\widehat{\SSet})$ (here we identify $|U|$ with its image under the natural map of sets $|j|:|U|\rightarrow |X|$ which is injective). 
\end{definition}

\begin{remark}\label{compatible underlying topological spaces}
	In the special case where the functor $X$ is representable by a quasi-separated spectral algebraic space, it follows from \cite[3.6.3.1]{lurie2018sag} that the topological space $|X|$ defined above is homeomorphic to the topological space associated to $X$ in the sense of \cite[3.6.1.1]{lurie2018sag}.
\end{remark}

\begin{pg}
	For later reference, we record some observations whose proofs are immediate:
\end{pg}

\begin{lemma}\label{basic properties of the underlying topological space}
	Let $f:X\rightarrow Y$ be a morphism in $\Fun(\CAlg^{\cn},\widehat{\SSet})$, where $X$ and $Y$ satisfy condition $(\ast)$ of \emph{\ref{the induced topological space of quasi-geometric stacks}}. Then:
\begin{enumerate}[$(i)$]
\item The induced map of sets $|f|:|X|\rightarrow |Y|$ is continuous. 
\item If $f$ is representable, then it is surjective if and only if the induced map $|f|$ is surjective. 
\item Suppose we are given a pullback diagram
$$
\Pull{X'}{Y'}{X}{Y}{}{}{}{}
$$
of functors satisfying condition $(\ast)$ of \emph{\ref{the induced topological space of quasi-geometric stacks}}. Then the induced map $|X'|\rightarrow |X|\times_{|Y|}|Y'|$ is a surjection of topological spaces.
\end{enumerate}
\end{lemma}

\begin{pg}
	The rest of this section is devoted to investigating some properties of the ``underlying topological spaces" of \ref{the induced topological space of quasi-geometric stacks}. 
\end{pg}

\begin{lemma}\label{quotient topology}
	Let $\pi:\sfX_0\rightarrow X$ be a morphism of functors appearing in condition $(\ast)$ of \emph{\ref{the induced topological space of quasi-geometric stacks}}. Then the induced map of topological spaces $|\pi|:|\sfX_0|\rightarrow |X|$ is open. In particular, $|\pi|$ is a quotient map.
\end{lemma}
\begin{proof}
	Let $U\subseteq |\sfX_0|$ be an open subset; we wish to show that its image under $|\pi|$ is open. Let $X_U$ be the subfunctor of $X$ which carries each object $A\in \CAlg^{\cn}$ to the summand of $X(A)$ spanned by those $\eta\in X(A)$ for which the induced map $|\eta|:|\Spec A|\rightarrow |X|$ factors through $|\pi|(U)$. We claim that the inclusion $j:X_U\rightarrow X$ is representable open. For this, let $\eta:\Spec A \rightarrow X$ be a point and consider a pullback diagram of functors
$$
\Pull{\Spec A \times_X\sfX_0}{\sfX_0}{\Spec A}{X.}{\eta'}{\pi'}{\pi}{\eta}
$$
Using \ref{basic properties of the underlying topological space}, we can identify $|\eta|^{-1}(|\pi|(U))$ with $|\pi'|(|\eta'|^{-1}U)$; in particular, it is an open subset of $|\Spec A|$ because $|\pi'|$ is an open map by virtue of a refinement of \cite[4.3.4.3]{lurie2018sag} without the quasi-compact assumption 
(which can be proven with little additional effort). 
Then \cite[19.2.4.1]{lurie2018sag} guarantees that $j$ is representable open. By construction, we have that $|j|(|X_U|)=|\pi|(U)$, thereby completing the proof.
\end{proof}

\begin{lemma}\label{flat and locally almost of finite presentation from affines}
	Let $X:\CAlg^{\cn}\rightarrow \widehat{\SSet}$ be a functor satisfying condition $(\ast)$ of \emph{\ref{the induced topological space of quasi-geometric stacks}} and let $\eta: \Spec A \rightarrow X$ be a representable morphism which is flat and locally almost of finite presentation. Then $|\eta|(|\Spec A|)\subseteq |X|$ is open.
\end{lemma}
\begin{proof}
	By virtue of \ref{quotient topology}, $|\pi|$ is a quotient map. Then the desired result follows by combining \ref{basic properties of the underlying topological space} with the variant of \cite[4.3.4.3]{lurie2018sag} mentioned in the proof of \ref{quotient topology}.	
\end{proof}

\begin{lemma}\label{basis of quasi-compact open subsets}
	Let $X:\CAlg^{\cn}\rightarrow \widehat{\SSet}$ be a functor satisfying condition $(\ast)$ of \emph{\ref{the induced topological space of quasi-geometric stacks}}. Then the underlying topological space $|X|$ has a basis consisting of quasi-compact open subsets of the form $|\eta|(|\Spec A|)$, where $A$ is a connective $\bbE_\infty$-ring and $\eta:\Spec A \rightarrow X$ is a relative spectral algebraic space which is flat and locally almost of finite presentation. 
\end{lemma}
\begin{proof}
	Using \ref{quotient topology} and \ref{flat and locally almost of finite presentation from affines}, we can reduce to the case where $X$ is a quasi-separated spectral algebraic space, in which case the desired result follows from (the proof of) \cite[3.6.3.3]{lurie2018sag}.
\end{proof}

\begin{pg}
	Combining 	\ref{flat and locally almost of finite presentation from affines} with \ref{basis of quasi-compact open subsets}, we immediately deduce the following generalization of \ref{flat and locally almost of finite presentation from affines}:
\end{pg}	
	
\begin{lemma}\label{flat and locally almost of finite presentation imply open map}
	Let $f: X'\rightarrow X$ be a morphism in $\Fun(\CAlg^{\cn}, \widehat{\SSet})$, where $X'$ and $X$ satisfy condition $(\ast)$ of \emph{\ref{the induced topological space of quasi-geometric stacks}}. If $f$ is representable flat and locally almost of finite presentation, then the induced map of topological spaces $|X'| \rightarrow |X|$ is open.
\end{lemma}

\begin{pg}
	Let $X$ be a functor which satisfies condition $(\ast)$ of \ref{the induced topological space of quasi-geometric stacks}. Under mild hypotheses, giving an open subset of $|X|$ is equivalent to giving an open immersion $U\rightarrow X$:
\end{pg}

\begin{lemma}\label{quasi-compact open subsets and quasi-geometric stacks}
	Let $X:\CAlg^{\cn}\rightarrow \widehat{\SSet}$ be a functor which satisfies condition $(\ast)$ of \emph{\ref{the induced topological space of quasi-geometric stacks}} and descent for the fpqc topology. Let $\pi:\sfX_0\rightarrow X$ be a morphism as in condition $(\ast)$ of \emph{\ref{the induced topological space of quasi-geometric stacks}}. Assume that the diagonal of $X$ is representable quasi-affine. If $\calU\subseteq |X|$ is a quasi-compact open subset, then there exist a quasi-geometric stack $U$ and a representable open immersion $j:U\rightarrow X$ such that $|j|(|U|)=\calU$.
\end{lemma}
\begin{proof}
	Let $U$ be the subfunctor of $X$ which carries an object $A\in \CAlg^{\cn}$ to the summand of $X(A)$ spanned by those $\eta\in X(A)$ for which the induced map of topological spaces $|\eta|:|\Spec A|\rightarrow |X|$ factors through $\calU$. It follows immediately that the inclusion $j:U\rightarrow X$ is representable open and that $|j|(|U|)=\calU$ (note that $U$ satisfies condition $(\ast)$ of \ref{the induced topological space of quasi-geometric stacks}). Using \cite[6.3.3.8]{lurie2018sag}, we see that $U$ is a sheaf for the fpqc topology. Since the diagonal of $X$ is quasi-affine, so is the diagonal of $U$. By virtue of \ref{flat and locally almost of finite presentation imply open map}, $|U|$ is homeomorphic to $\calU$, hence quasi-compact. Then \ref{basis of quasi-compact open subsets} guarantees that there exists a relative spectral algebraic space $\Spec A \rightarrow U$ which is faithfully flat (and locally almost of finite presentation), which completes the proof.
\end{proof}

\begin{pg}
	We now extend the relationship between reduced closed substacks of a (quasi-geometric) spectral Deligne-Mumford stack $\sfX$ and open subsets of $|\sfX|$ to quasi-geometric spectral algebraic stacks; see \cite[3.1.6.3]{lurie2018sag}.
\end{pg}

\begin{lemma}\label{the equivalent conditions of being reduced}
	Let $X$ be a quasi-geometric spectral algebraic stack. The following conditions are equivalent:
\begin{enumerate}[$(i)$]
\item For every fiber smooth morphism $f:\Spec A \rightarrow X$, the $\bbE_\infty$-ring $A$ is discrete and reduced.
\item There exists a fiber smooth surjection $\Spec A \rightarrow X$, where the $\bbE_\infty$-ring $A$ is discrete and reduced.
\end{enumerate}
\end{lemma}
\begin{proof}
	According to \cite[2.8.3.9]{lurie2018sag}, the property of being a $0$-truncated spectral Deligne-Mumford stack is local with respect to the flat topology, so the desired equivalence follows from the fact that for ordinary algebraic spaces, the property of being reduced is local with respect to the smooth topology (see, for example, \cite[\href{https://stacks.math.columbia.edu/tag/034E}{Tag 034E}]{stacks-project}).
\end{proof}

\begin{definition}
	Let $X$ be a quasi-geometric spectral algebraic stack. Let us say that $X$ is \emph{reduced} if it satisfies the equivalent conditions of \ref{the equivalent conditions of being reduced}.
\end{definition}

\begin{proposition}\label{reduced closed substacks of quasi-geometric spectral algebraic stacks}
	Let $j:U\rightarrow X$ be a representable open immersion of quasi-geometric spectral algebraic stacks. Then there exist a reduced quasi-geometric spectral algebraic stack $K$ and a representable closed immersion $i:K\rightarrow X$ such that $|i||K|=|X|-|j||U|$.
\end{proposition}
\begin{proof}
	Choose a fiber smooth surjection $f:\sfX_0\rightarrow X$, where $\sfX_0$ is affine. Let $\sfX_\bullet$ denote the \Cech nerve of the morphism $f$, which is a simplicial object of the $\infty$-category $\SpDM$ of spectral Deligne-Mumford stacks. The projections $U\times_X\sfX_n \rightarrow \sfX_n$ are open immersions of spectral Deligne-Mumford stacks, so that there is a simplicial object $\sfK_\bullet$ of $\SpDM$, where each $\sfK_n$ is the reduced closed substack complementary to $U\times_X\sfX_n$, and is quasi-geometric. Let $K$ denote the geometric realization of $\sfK_\bullet$ in the $\infty$-category of fpqc sheaves; it is a quasi-geometric stack by virtue of \cite[9.1.1.5]{lurie2018sag}. By construction, the diagram of fpqc sheaves
$$
\Pull{\sfK_0}{\sfX_0}{K}{X}{}{}{}{}
$$
is a pullback square, from which it follows immediately that $K$ is a reduced quasi-geometric spectral algebraic stack. Applying \cite[9.1.1.3]{lurie2018sag} to the diagram above, we deduce that the canonical morphism $K\rightarrow X$ is representable quasi-affine, thereby a closed immersion with the property that $|K|$ is complementary to $|U|$ (regarded as subsets of $|X|$).
\end{proof}

\section{Excision Squares}\label{Sec: Excision Squares}
	Our goal in this section is to supply a special presentation of quasi-geometric spectral algebraic stacks in the spirit of ``induction principle" for ordinary algebraic stacks; see \cite[4.1]{MR3754421} and \cite[Theorem E]{MR3754421}. For this, we introduce excision squares and stacky scallop decompositions of such stacks. 

\begin{definition}\label{excision squares of quasi-geometric stacks}
	A diagram of quasi-geometric spectral algebraic stacks $\sigma:$
$$
\Pull{U'}{X'}{U}{X}{}{}{f}{j}
$$
is an \emph{excision square} if it satisfies the following conditions:
\begin{enumerate}[(i)]
\item The diagram $\sigma$ is a pullback square.
\item The morphism $j$ is a representable open immersion.
\item The morphism $f$ is representable \'etale.
\item The projection $K\times_X X' \rightarrow K$ is an equivalence (here $K$ denotes the reduced closed substack of $X$ complementary to $U$; see \ref{reduced closed substacks of quasi-geometric spectral algebraic stacks}). 
\end{enumerate}
\end{definition}

\begin{remark}\label{reduced closed substacks and excision squares of spectral Deligne-Mumford stacks} 
	According to \cite[p.321]{lurie2018sag}, a diagram of spectral Deligne-Mumford stacks 
$$
\Pull{\sfU'}{\sfX'}{\sfU}{\sfX}{j'}{f'}{}{}
$$
is an \emph{excision square} if it is a pushout square, $j'$ is an open immersion, and $f'$ is \'etale. If it is a diagram of quasi-geometric spectral Deligne-Mumford stacks, then it is an excision square in the sense of \cite[p.321]{lurie2018sag} if and only if the associated square of quasi-geometric stacks is an excision square in the sense of \ref{excision squares of quasi-geometric stacks} (see also \cite[9.1.4.4]{lurie2018sag}). 
\end{remark}

\begin{pg}
	Let $\sfX$ be a spectral Deligne-Mumford stack. According to \cite[2.5.3.1]{lurie2018sag}, a \emph{scallop decomposition} of $\sfX$ consists of a sequence of open immersions $\emptyset \simeq \sfU_0\rightarrow \sfU_1 \rightarrow \cdots \rightarrow \sfU_n \simeq \sfX$ such that for each $1\leq i \leq n$, there exists an excision square of spectral Deligne-Mumford stacks
$$
\Pull{\sfV}{\sfY}{\sfU_{i-1}}{\sfU_i,}{}{}{}{}
$$
where $\sfY$ is affine and $\sfV$ is quasi-compact. 
This is a useful device for proving many basic results in the theory of spectral algebraic geometry by reducing to the affine case. However, a spectral Deligne-Mumford stack admits a scallop decomposition if and only if it is a quasi-compact quasi-separated spectral algebraic space (see \cite[3.4.2.1]{lurie2018sag}), so that the concept of a scallop decomposition is not adequate for spectral Deligne-Mumford stacks which are not spectral algebraic spaces. To incorporate a wider class of spectral algebro-geometric objects, we should relax the requirement that $\sfY$ is affine in the diagram above; we therefore allow $\sfY$ to be quasi-geometric spectral algebraic stacks, which is sufficient for our needs in this paper:
\end{pg}

\begin{definition}\label{stacky scallop decompositions}
	Let $X$ be a quasi-geometric spectral algebraic stack. A \emph{stacky scallop decomposition} of $X$ consists of a sequence of representable open immersions of quasi-geometric spectral algebraic stacks
$$
\emptyset \simeq U_0\rightarrow U_1 \rightarrow \cdots \rightarrow U_n \simeq X
$$
satisfying the following condition: for each $1\leq i \leq n$, there exists an excision square 
$$
\Pull{V}{W}{U_{i-1}}{U_i}{}{}{}{}
$$
of quasi-geometric spectral algebraic stacks (see \ref{excision squares of quasi-geometric stacks}).
\end{definition}

\begin{pg}
	The notion of \emph{Nisnevich covering} of quasi-compact quasi-separated spectral algebraic space (see \cite[3.7.1.1]{lurie2018sag}) admits a straightforward extension to quasi-geometric spectral algebraic stacks:
\end{pg}

\begin{definition}\label{Nisnevich covering of quasi-geometric spectral algebraic stacks}
	Let $X$ be a quasi-geometric spectral algebraic stack. Let $\{p_\alpha: W_\alpha \rightarrow X\}$ be a collection of representable \'etale morphisms of quasi-geometric spectral algebraic stacks. We say that $\{p_\alpha\}$ is a \emph{Nisnevich covering} of $X$ if there exists a sequence of open immersions of quasi-geometric spectral algebraic stacks 
$$
\emptyset \simeq U_{n+1} \hookrightarrow \cdots \hookrightarrow U_0 \simeq X
$$
satisfying the following condition: for each $0\leq i \leq n$, let $K_i$ denote the reduced closed substack of $U_i$ which is complementary to $U_{i+1}$ (see \ref{reduced closed substacks of quasi-geometric spectral algebraic stacks}). Then the composition $K_i\rightarrow U_i \rightarrow X$ factors through some $p_\alpha$. 
\end{definition}

\begin{lemma}\label{stacky scallop decomposition induced by a Nisnevich covering}
	Let $p:W \rightarrow X$ be a Nisnevich covering of quasi-geometric spectral algebraic stacks. Then $p$ induces a stacky scallop decomposition of $X$. 
\end{lemma}
\begin{proof}
	In the situation of \ref{Nisnevich covering of quasi-geometric spectral algebraic stacks}, for each $0\leq m \leq n$, consider a subset of $|U_m\times_XW|$ which is complementary to $|i_m|(|K_m\times_XW|-|s_m||K_m|)$, where $i_m:K_m\times_XW\rightarrow U_m\times_XW$ is the closed immersion determined by $K_m\rightarrow U_m$ and $s_m$ is a section of the projection $K_m\times_XW \rightarrow K_m$ (note that $p$ is a Nisnevich covering). Since $s_m$ is an open immersion, this subset is open. Moreover, it is quasi-compact because it can be written as a disjoint union of the image of the map $|i_m\circ s_m|$ and $|U_{m+1}\times_XW|$. According to \ref{quasi-compact open subsets and quasi-geometric stacks}, this quasi-compact open subset determines a representable open immersion $W_m\rightarrow U_m\times_XW$ of quasi-geometric spectral algebraic stacks. Composing this with the projection to $U_m$, we obtain a representable \'etale morphism $W_m\rightarrow U_m$. Note that the composition $i_m\circ s_m$ factors through $W_m$, inducing a section of the projection $K_m\times_{U_m}W_m\rightarrow K_m$. By construction, this section is a surjective open immersion, hence an equivalence. Consequently, the pullback square of quasi-geometric spectral algebraic stacks
$$
\Pull{U_{m+1}\times_{U_m}W_m}{W_m}{U_{m+1}}{U_m}{}{}{}{}
$$
is an excision square, thereby completing the proof.
\end{proof}

\begin{pg}
	Our primary goal in this section is to produce some presentation of quasi-geometric algebraic stacks, which allows us to apply some d\'evissage method for the study of those stacks. To obtain such a presentation, we will make use of the Hilbert functors in the setting of spectral algebraic geometry. Note that \cite[8.3.3]{MR2717174} shows the representability of the Hilbert functors in the derived setting; we will prove a similar result for the spectral Hilbert functors. We begin by defining the Hilbert functors in the spectral setting. Let $p:X\rightarrow S$ be a representable morphism in $\Fun(\CAlg^{\cn}, \widehat{\SSet})$. Let $\pi:\CAlg^{\cn}_S\rightarrow \CAlg^{\cn}$ be a left fibration classified by $S$ (see \cite[3.3.2.2]{MR2522659}). Let us identify objects of $\CAlg^{\cn}_S$ with pairs $(A, \eta)$, where $A$ is a connective $\bbE_\infty$-ring and $\eta\in S(A)$ is an $A$-valued point of $S$. Note that the opposite of $\CAlg^{\cn}_S$ can be identified with the fiber product $(\CAlg^{\cn})^{\op}\times_{\Fun(\CAlg^{\cn}, \widehat{\SSet})} \Fun(\CAlg^{\cn}, \widehat{\SSet})_{/S}$, where $(\CAlg^{\cn})^{\op} \rightarrow \Fun(\CAlg^{\cn}, \SSet)$ is the Yoneda embedding. Consider the composition
$$
(\CAlg^{\cn}_S)^{\op} \subseteq \Fun(\CAlg^{\cn}, \widehat{\SSet})_{/S} \rightarrow \Fun(\CAlg^{\cn}, \widehat{\SSet})_{/X}\rightarrow \Fun(\CAlg^{\cn}, \widehat{\SSet}),
$$ 
where the middle arrow is the base change functor $S'\mapsto S'\times_SX$ and the last is the forgetful functor. This composition can be described more informally as follows: to each pair $(A, \eta)$, it assigns the fiber product $\Spec A\times_SX$, where $\Spec A \rightarrow S$ is determined by $\eta$. We also consider the composition $\Fun(\Delta^1, \SpDM)\rightarrow \Fun(\{1\}, \SpDM)\rightarrow \Fun(\CAlg^{\cn}, \widehat{\SSet})$, where the first map is an evaluation at $\{1\}\subseteq \Delta^1$ and the second is the fully faithful embedding. Let $\calC$ denote the full subcategory of the fiber product 
$$
(\CAlg^{\cn}_S)^{\op}\times_{\Fun(\CAlg^{\cn}, \widehat{\SSet})}\Fun(\Delta^1, \SpDM)
$$
spanned by those morphisms $f:\sfY \rightarrow \Spec A\times_SX$, where $\sfY$ is a spectral Deligne-Mumford stack, $f$ is a closed immersion, and the composition of $f$ with the projection $\Spec A\times_SX \rightarrow \Spec A$ is proper, flat, and locally almost of finite presentation. Let $\widehat{\Hilb}_{X/S}:\CAlg^{\cn}  \rightarrow \widehat{\Cat}_\infty$ denote the functor classifying the Cartesian fibration $\calC \rightarrow (\CAlg^{\cn}_S)^{\op}\stackrel{\pi^{\op}}{\rightarrow} (\CAlg^{\cn})^{\op}$ (here $\widehat{\Cat}_\infty$ denotes the $\infty$-category of (not necessarily small) $\infty$-categories; see \cite[3.0.0.5]{MR2522659}). Let $\Hilb_{X/S}:\CAlg^{\cn} \rightarrow \widehat{\SSet}$ be the functor given by the formula $\Hilb_{X/S}(A)=\widehat{\Hilb}_{X/S}(A)^\simeq$, where $\widehat{\Hilb}_{X/S}(A)^\simeq$ denotes the largest Kan complex contained in $\widehat{\Hilb}_{X/S}(A)$. Note that there is a canonical morphism of functors $\Hilb_{X/S}\rightarrow S$. 
\end{pg}

\begin{theorem}\label{spectral Hilbert functor}
	Let $p:X\rightarrow S$ be a morphism in $\Fun(\CAlg^{\cn},\widehat{\SSet})$ which is representable, separated, and locally almost of finite presentation. Then the canonical morphism $\Hilb_{X/S}\rightarrow S$ is a relative spectral algebraic space which is locally almost of finite presentation.
\end{theorem}
\begin{proof}
	We will use the criterion for representability supplied by \cite[18.1.0.2]{lurie2018sag}. The canonical morphism $\Hilb_{X/S}\rightarrow S$ is infinitesimally cohesive and nilcomplete (see \cite[17.3.7.1]{lurie2018sag}) by virtue of \cite[16.3.0.1, 16.3.2.1]{lurie2018sag} and  \cite[19.4.1.2, 19.4.2.3]{lurie2018sag}, respectively. 

We next show that the morphism $\Hilb_{X/S}\rightarrow S$ admits a relative cotangent complex of \cite[17.2.4.2]{lurie2018sag}. We will prove this by verifying conditions $(a)$ and $(b)$ of \cite[17.2.4.3]{lurie2018sag}. Let $A$ be a connective $\bbE_\infty$-ring and let $\eta\in \Hilb_{X/S}(A)$ be a point corresponding to a pair $(\zeta, i:\sfY\rightarrow X\times_S\Spec A)$, where $\zeta \in S(A)$ is a point and $i:\sfY \rightarrow X\times_S\Spec A$ is a closed immersion of spectral Deligne-Mumford stacks for which the composition $f:\sfY\rightarrow \Spec A\times_SX \rightarrow \Spec A$ is proper, flat, and locally almost of finite presentation. Let $F:\Mod^{\cn}_A \rightarrow \SSet$ be the functor defined by the formula 
$$
F(M)=\fib(\Hilb_{X/S}(A\oplus M) \rightarrow \Hilb_{X/S}(A)\times_{S(A)}S(A\oplus M)),
$$ 
where the fiber is taken over the point of $\Hilb_{X/S}(A)\times_{S(A)}S(A\oplus M)$ determined by $\eta$. We wish to show that $F$ is corepresented by an almost connective $A$-module. According to \cite[19.4.3.1]{lurie2018sag}, the fiber is canonically equivalent to $\Map_{\QCoh(\sfY)}(L_{\sfY/X\times_S\Spec A}, \Sigma f^\ast M)$. By virtue of \cite[6.4.5.3]{lurie2018sag}, $f^\ast:\QCoh(\Spec A)\rightarrow \QCoh(Y)$ admits a left adjoint $f_+$, so that $F(M)$ is corepresented by an $A$-module $\Sigma^{-1}f_+L_{\sfY/X\times_S\Spec A}$. Since $i$ is a closed immersion, it follows from \cite[17.1.4.3]{lurie2018sag} that $L_{\sfY/X\times_S\Spec A}$ is $1$-connective, so the $A$-module $\Sigma^{-1}f_+L_{\sfY/X\times_S\Spec A}$ is connective as desired (here we use the fact that $f$ is flat). Condition $(b)$ is an immediate consequence of \cite[6.4.5.4]{lurie2018sag}. We note that $L_{\sfY/X\times_S\Spec A}$ is almost perfect because $i$ is locally almost of finite presentation (see \cite[17.1.5.1]{lurie2018sag}), so that $\Sigma^{-1}f_+L_{\sfY/X\times_S\Spec A}$ is almost perfect by virtue of \cite[6.4.5.2]{lurie2018sag} and \cite[7.2.4.11]{lurie2017ha}. We conclude that the relative cotangent complex $L_{\Hilb_{X/S}/S}$ is not only connective, but also almost perfect. 

We now show that $\Hilb_{X/S}\rightarrow S$ is a relative spectral algebraic space. Since the formation of Hilbert functors is compatible with base change, we may assume that $S$ is an affine spectral Deligne-Mumford stack. We wish to show that $\Hilb_{X/S}$ is representable by a spectral algebraic space. Let $\CAlg^{\heartsuit}$ denote the $\infty$-category of discrete $\bbE_\infty$-rings, which can be identified with the nerve of the category of commutative rings; see \cite[7.1.0.3]{lurie2017ha}. The restriction of $\Hilb_{X/S}$ to $\CAlg^\heartsuit$ is equivalent to the ordinary Hilbert functor associated to the morphism of ordinary algebraic spaces $\tau_{\leq 0}X \rightarrow \tau_{\leq 0}S$ (here we use the fact that $p$ is a relative spectral algebraic space; see \cite[3.2.1.1]{lurie2018sag}), which is representable by an ordinary algebraic space; see \cite[6.2]{MR0260746}. Since $S$ is assumed to be representable, it admits a cotangent complex, infinitesimally cohesive, and nilcomplete by virtue of \cite[17.2.5.4, 17.3.1.2, 17.3.2.3]{lurie2018sag}. Combining \cite[17.3.7.3]{lurie2018sag} and \cite[17.3.9.1]{lurie2018sag} with the above discussion, we deduce that $\Hilb_{X/S}$ satisfies the hypothesis of \cite[18.1.0.2]{lurie2018sag}, and is therefore representable by a spectral algebraic space as desired.

It remains to prove that the morphism $\Hilb_{X/S}\rightarrow S$ is locally almost of finite presentation. We may assume that $S$ is affine. Using \cite[19.4.2.3]{lurie2018sag}, we may further assume that $S$ is $0$-truncated. We have already seen that $\Hilb_{X/S}\rightarrow S$ is infinitesimally cohesive and admits a relative cotangent complex which is almost perfect. By virtue of \cite[17.4.2.2]{lurie2018sag}, it will suffice to check condition $(\ast)$ of \cite[17.4.2.1]{lurie2018sag}: for every filtered diagram $\{A_\alpha\}$ of commutative rings having colimit $A$, the canonical map
$$
\colim \Hilb_{X/S}(A_\alpha)\rightarrow \colim S(A_\alpha) \times_{S(A)}\Hilb_{X/S}(A)
$$
is an equivalence. Since the restrictions of $\Hilb_{X/S}$ and $\Hilb_{\tau_{\leq 0}X/S}$ to $\CAlg^\heartsuit$ are equivalent, we can reduce to the case where $p:X\to S$ is a morphism of ordinary algebraic spaces, in which case the desired result follows from its classical counterpart (see \cite[6.2]{MR0260746} and \cite[8.14.2]{MR0217086}).
\end{proof}

\begin{remark}\label{the subfunctor of Hilbert functor classifying clopen immersions}
	Let $\Hilb_{X/S}^{\et}\subseteq \Hilb_{X/S}$ be the subfunctor which carries an $\bbE_\infty$-ring $A$ to the summand of $\Hilb_{X/S}(A)$ spanned by those pairs $(A, i:\sfY\rightarrow \Spec A \times_SX)$ for which $i$ is \'etale. Under the additional assumption that the morphism $p:X\rightarrow S$ is flat, a similar argument shows that the canonical morphism $\Hilb_{X/S}^{\et}\rightarrow S$ is a relative spectral algebraic space which is locally almost of finite presentation; moreover, it is \'etale by construction (here we use the fact that a morphism of spectral Deligne-Mumford stacks which is locally almost of finite presentation is \'etale if and only if its relative cotangent complex vanishes; see \cite[17.1.5.1]{lurie2018sag}) and is separated by reducing to its classical counterpart (see \cite[6.1]{MR0260746}). 
\end{remark}

\begin{pg}\label{degree of fibers}
	Our proof of \ref{a presentation for applying quasi-finite devissage} will make use of the notion of degree of fibers defined as follows: let $f:X\rightarrow Y$ be a representable flat, quasi-compact, separated, and locally quasi-finite morphism in $\Fun(\CAlg^{\cn},\widehat{\SSet})$, where $X$ and $Y$ satisfy condition $(\ast)$ of \ref{the induced topological space of quasi-geometric stacks}. Suppose we are given a point $\eta: \Spec \kappa \rightarrow Y$ which represents some $y\in |Y|$. The projection $\Spec \kappa \times_YX \rightarrow \Spec \kappa$, which can be identified with a morphism of ordinary schemes, is finite flat of degree $d$ for some $d\geq 0$; this integer does not depend on the choice of $\eta$. We therefore obtain a well-defined map $n_{X/Y}:|Y| \rightarrow \mathbb{Z}_{\geq 0}$ which carries $y\in |Y|$ to the degree of finite flat morphism $\Spec \kappa \times_YX \rightarrow \Spec \kappa$ determined by any point $\Spec \kappa \rightarrow Y$ representing $y$. 

We are now ready to prove an analogue of \cite[4.1]{MR3754421} in spectral algebraic geometry:
\end{pg}

\begin{theorem}\label{a presentation for applying quasi-finite devissage}
	Let $X:\CAlg^{\cn} \rightarrow \widehat{\SSet}$ be a quasi-geometric spectral algebraic stack which admits a quasi-finite presentation (see \emph{\ref{a quasi-finite presentation}}). Then there exist morphisms of quasi-geometric spectral algebraic stacks $p:W\rightarrow X$ and $q:\sfV \rightarrow W$ such that $p$ is a separated Nisnevich covering, $\sfV$ is a quasi-affine spectral Deligne-Mumford stack, and $q$ is representable finite, faithfully flat, and locally almost of finite presentation. 
\end{theorem}

\begin{proof}
	Using our assumption that $X$ admits a quasi-finite presentation, we can choose a connective $\bbE_\infty$-ring $A$ and a morphism $f:\Spec A \rightarrow X$ which is locally quasi-finite, faithfully flat, and locally almost of finite presentation. Choose a fiber smooth surjection $g:\sfY\rightarrow X$, where $\sfY$ is an affine spectral Deligne-Mumford stack. Consider a pullback square of quasi-geometric spectral algebraic stacks
$$
\Pull{\sfY'}{\Spec A}{\sfY}{X.}{g'}{f'}{f}{g}
$$
The underlying morphism $\tau_{\leq 0}f':\tau_{\leq 0}\sfY'\rightarrow \tau_{\leq 0}\sfY$ of ordinary of algebraic spaces is quasi-affine, locally quasi-finite, faithfully flat, and locally of finite presentation. Combining \cite[\href{https://stacks.math.columbia.edu/tag/07RZ}{Tag 07RZ}]{stacks-project} with \cite[\href{https://stacks.math.columbia.edu/tag/03JA}{Tag 03JA}]{stacks-project}, we deduce that there exists a sequence of quasi-compact open immersions of spectral algebraic spaces 
$$
\emptyset\simeq \sfV_{n+1} \hookrightarrow \cdots \hookrightarrow \sfV_0 \simeq \sfY
$$
with the following properties:
\begin{enumerate}[(i)]
\item For every $0\leq i\leq n+1$, we have $|\sfV_i|=\{y\in |Y|: n_{Y'/Y}(y)\geq i \}$ (see \ref{degree of fibers}). 
\item For each $0\leq i\leq n$, let $\sfK_i$ denote the reduced closed substack of $\sfV_i$ complementary to $\sfV_{i+1}$ (see \cite[3.1.6.3]{lurie2018sag}). Then the projection $\sfK_i\times_{\sfY}\sfY'\rightarrow \sfK_i$ is finite flat of degree $i$. 
\end{enumerate}
Since $g$ is flat and locally almost of finite presentation, \ref{flat and locally almost of finite presentation imply open map} guarantees that for each $i$, the image of $|\sfV_i|$ under $|g|$ is quasi-compact open, and therefore gives rise to an open immersion $U_i\rightarrow X$ of quasi-geometric spectral algebraic stacks (see \ref{quasi-compact open subsets and quasi-geometric stacks}). We claim that the sequence of open immersions of quasi-geometric spectral algebraic stacks 
$$
\emptyset \simeq U_{n+1} \hookrightarrow \cdots \hookrightarrow U_0 \simeq X
$$
gives a stacky scallop decomposition of $X$. Let $K'_i$ denote the reduced closed substack of $U_i$ complementary to $U_{i+1}$; see \ref{reduced closed substacks of quasi-geometric spectral algebraic stacks}. Using the description of $|\sfV_i|$, we see that the canonical morphism $\sfV_i \rightarrow U_i\times_X\sfY$ is an equivalence. In particular, the induced morphism $\sfK_i\rightarrow K'_i$ can be identified with a pullback of $\sfV_i\rightarrow U_i$, and therefore is a flat covering of \cite[2.8.3.1]{lurie2018sag}. It then follows from \cite[5.2.3.5]{lurie2018sag} that the projection $K'_i\times_X\Spec A\rightarrow K'_i$ is finite flat of degree $i$, so that the identity morphism on $K'_i\times_X\Spec A$ induces a factorization of the immersion $K'_i\rightarrow X$ through the subfunctor $\Hilb_{\Spec A/X}^{\et} \subseteq \Hilb_{\Spec A/X}$ of \ref{the subfunctor of Hilbert functor classifying clopen immersions}. 
Using \ref{basis of quasi-compact open subsets}, we can choose a quasi-compact open subset $\calW' \subseteq|\Hilb_{\Spec A/X}^{\et}|$ which contains the image of $|K'_i|$ in $|\Hilb_{\Spec A/X}^{\et}|$ for all $i$. Let $j:W'\rightarrow \Hilb_{\Spec A/X}^{\et}$ be a representable open immersion such that $|j||W'|=\calW'$ (see the proof of \ref{quasi-compact open subsets and quasi-geometric stacks}). Let $p'$ denote the composition $W' \rightarrow \Hilb_{\Spec A/X}^{\et} \rightarrow X$. Note that \ref{the subfunctor of Hilbert functor classifying clopen immersions} guarantees that $p'$ is representable. Combining this observation with \cite[6.3.3.8]{lurie2018sag}, we see that $W'$ is a sheaf for the fpqc topology. Since the projection $\sfY\times_XW'\rightarrow W'$ is a fiber smooth surjection and $|W'|$ is quasi-compact, \ref{basis of quasi-compact open subsets} (and its proof) supplies a fiber smooth surjection $\pi:\Spec B \rightarrow W'$, where $B$ is a connective $\bbE_\infty$-ring. Note that $\pi$ is quasi-affine because $p'$ is separated and the composition $p' \circ\pi$ is quasi-affine (here we use the fact that the diagonal of $X$ is quasi-affine), so that the diagonal of $W'$ is representable quasi-affine by virtue of \cite[9.1.1.2]{lurie2018sag}. Consequently, we conclude that $W'$ is a quasi-geometric spectral algebraic stack. In particular, $p'$ is quasi-compact. Combining this observation with the fact that $p'$ is \'etale and separated (see \ref{the subfunctor of Hilbert functor classifying clopen immersions}), we deduce that it is quasi-affine by virtue of \cite[3.3.0.2]{lurie2018sag}. Let $i:V\rightarrow W'\times_X\Spec A$ denote the clopen immersion of quasi-geometric spectral algebraic stacks (see \cite[3.1.7.2]{lurie2018sag}) determined by the inclusion $W'\rightarrow \Hilb_{\Spec A/X}^{\et}$. Let $q'$ denote the composition of $i$ with the projection $W'\times_X\Spec A \rightarrow W'$. Since $q'$ is flat and locally almost of finite presentation, \ref{flat and locally almost of finite presentation imply open map} guarantees that the image of $|V|$ under $|q'|$ is quasi-compact open, and therefore induces an open immersion $W\rightarrow W'$ of quasi-geometric spectral algebraic stacks by virtue of \ref{quasi-compact open subsets and quasi-geometric stacks}. Shrinking $|W'|$ to the image of $|q'|$, we obtain a surjection $q:V\rightarrow W$ of quasi-geometric spectral algebraic stacks. Combining the fact that $i$ is clopen immersion with \cite[21.4.6.4]{lurie2018sag} and \cite[2.5.7.4]{lurie2018sag}, we observe that $q$ is flat, locally almost of finite presentation, and quasi-affine. Since $q'$ is proper, $q$ is also proper, hence finite by virtue of \cite[5.2.1.1]{lurie2018sag}. Invoking quasi-affineness of $p'$, we conclude that $V$ is representable by a quasi-affine Deligne-Mumford stack. Let $p$ denote the composition of the inclusion $W\subseteq W'$ with $p'$. We will complete the proof by showing that $p$ is a separated Nisnevich covering. Since each $K'_i$ factors through $W'$, it will suffice to show that the image of $|K'_i|$ in $|W'|$ is contained in $|q'||V|$. By construction, the projection $K_i'\times_X\Spec A \rightarrow K_i'$ is a pullback of $q'$, so the desired result follows by combining this observation with the fact that $f$ is surjective.
\end{proof}

\section{Twisted Compact Generations}\label{Sec: Twisted Compact Generations}
	In this section, we prove that quasi-geometric spectral algebraic stacks which admit a quasi-finite presentation are of twisted compact generation. This result will play a central role in our proof of \ref{Brauer spaces and Azumaya algebras for quasi-geometric spectral algebraic stacks}.

\begin{pg}\label{compact generators}
	To formulate the main definition of interest to us in this section (that is, \ref{twisted compact generation}), we recall a bit of terminology. According to \cite[5.5.7.1]{MR2522659}, an $\infty$-category $\calC$ is \emph{compactly generated} if it is presentable and $\omega$-accessible, or equivalently if the inclusion $\Ind(\calC^\omega) \rightarrow \calC$ is an equivalence of $\infty$-categories, where $\calC^\omega \subseteq \calC$ is the full subcategory spanned by the compact objects of $\calC$ (here $\Ind(\calC^\omega)$ denotes the $\infty$-category of Ind-objects of $\calC^\omega$; see \cite[5.3.5.1]{MR2522659}). 

Now let $\calC$ be a presentable stable $\infty$-category, and let $\{C_i\}_{i \in I}$ be a collection of compact objects of $\calC$. We say that the collection $\{C_i\}$ is a \emph{set of compact generators} for $\calC$ if it satisfies the following condition: an object $C\in \calC$ is equivalent to $0$ if the graded abelian group $\Ext^\ast_{\calC}(C_i, C)$ is zero for all $i\in I$. Note that if $\calC$ is compactly generated, the collection of compact objects of $\calC$ forms a set of compact generators. 
\end{pg}

\begin{pg}\label{adjunction and a set of compact generators}
	Suppose we are given an adjunction $\Adjoint{L}{\calC}{\calD}{R}$ between presentable $\infty$-categories, where the right adjoint $R$ is conservative and preserves small filtered colimits. It follows from \cite[6.2]{lurie2011dagxi} that if $\calC$ is compactly generated, then so is $\calD$. Note that in the special case where $\calC$ and $\calD$ are presentable stable $\infty$-categories, if $\{C_i\}_{i\in I}$ is a set of compact generators for $\calC$, then $\{L(C_i)\}$ is a set of compact generators for $\calD$. 
\end{pg}

\begin{pg}\label{unbounded derived category}
	Let $X:\CAlg^{\cn}\rightarrow \widehat{\SSet}$ be a functor and let $\QCoh(X)$ denote the \emph{$\infty$-category of quasi-coherent sheaves on $X$} of \cite[6.2.2.1]{lurie2018sag}. More informally, we can think of an object $F\in \QCoh(X)$ as a rule which assigns to each connective $\bbE_\infty$-ring $R$ and each point $\eta\in X(R)$ an $R$-module $F_\eta\in \Mod_R$, which depends functorially on $R$ and $\eta$ (see \cite[6.2.2.7]{lurie2018sag}). According to \cite[6.2.6]{lurie2018sag}, the $\infty$-category $\QCoh(X)$ can be equipped with a symmetric monoidal structure, where the tensor product is given informally by the formula $(F\otimes F')_\eta \simeq F_\eta \otimes_R F'_\eta$ for each point $\eta\in X(R)$; let $\calO_X$ denote the unit object of $\QCoh(X)$. In the special case where $X$ is representable by an ordinary Deligne-Mumford stack $(\calX, \calO)$, let $\mathrm{D}(X)$ denote the derived $\infty$-category of the Grothendieck abelian category $\Mod_{\calO}$ of $\calO$-modules; see \cite[1.3.5.8]{lurie2017ha}. It then follows from \cite[2.2.6.2]{lurie2018sag} that there is a canonical equivalence $\QCoh(X)\simeq \mathrm{D}_{\mathrm{qc}}(\calX)$, where $\mathrm{D}_{\mathrm{qc}}(\calX) \subseteq \mathrm{D}(X)$ is the full subcategory spanned by those chain complexes of $\calO$-modules whose homologies are quasi-coherent.
\end{pg}

\begin{remark}\label{quasi-coherent sheaves in DAG}
	If $X':\CAlg^\Delta_R\rightarrow \widehat{\SSet}$ is a functor, we define the $\infty$-category $\QCoh(X')$ of quasi-coherent sheaves on $X'$ to be the $\infty$-category $\QCoh({\Theta_R}_!X')$ (here we regard ${\Theta_R}_!X'$ as an object of the slice $\infty$-category $\Fun(\CAlg^{\cn}, \widehat{\SSet})_{/R}$). A slight variant of \cite[6.2.3.1]{lurie2018sag} (using the fiber smooth topology in place of the fpqc topology) guarantees that for every quasi-geometric derived algebraic stack $X$ over $R$, the canonical map $\QCoh(X^\circ)\rightarrow \QCoh(X)$ is an equivalence of $\infty$-categories (here $X^\circ$ is regarded as an object of $\Fun(\CAlg^{\cn}, \widehat{\SSet})_{/R}$). 
\end{remark}

\begin{pg}\label{the infinity category of prestable quasi-coherent stacks}
	In \cite{lurie2018sag}, Lurie develops the theory of quasi-coherent stacks, which plays an analogous role of the categories of twisted sheaves in the setting of spectral algebraic geometry. In this analogy, the $\infty$-category of global sections of quasi-coherent stacks is an analogue of the derived category of twisted sheaves. We now give a quick review of some basic definitions and notations. Let $\LinCat^{\PSt}$ denote the $\infty$-category whose objects are pairs $(R, \calC)$, where $R$ is a connective $\bbE_\infty$-ring and $\calC$ is a prestable $R$-linear $\infty$-category of \cite[D.1.4.1]{lurie2018sag}. Let $\QStk^{\PSt}: \Fun(\CAlg^{\cn}, \widehat{\SSet})^{\op} \rightarrow \widehat{\Cat}_\infty$ denote the functor obtained by applying \cite[6.2.1.11]{lurie2018sag} to the projection $q: \LinCat^{\PSt}\rightarrow \CAlg^{\cn}$. Let $X:\CAlg^{\cn}\rightarrow \widehat{\SSet}$ be a functor. We refer to $\QStk^{\PSt}(X)$ as the \emph{$\infty$-category of prestable quasi-coherent stacks on $X$}; see \cite[10.1.2.4]{lurie2018sag}. More informally, an object $\calC\in \QStk^{\PSt}(X)$ is a rule which assigns to each connective $\bbE_\infty$-ring $R$ and each point $\eta\in X(R)$ a prestable $R$-linear $\infty$-category $\calC_\eta$, depending functorially on the pair $(R,\eta)$ (see \cite[10.1.1.3]{lurie2018sag} for more details). 
\end{pg}

\begin{pg}
	Let $\QStk^{\PSt}_R: \Fun(\CAlg^{\cn}_R, \widehat{\SSet})^{\op} \rightarrow \widehat{\Cat}_\infty$ denote the functor obtained by applying \cite[6.2.1.11]{lurie2018sag} to the projection $q_R:\CAlg^{\cn}_R\times_{\CAlg^{\cn}} \LinCat^{\PSt}\rightarrow \CAlg^{\cn}_R$. Note that if $X$ is an image of $X'$ under the equivalence $\Fun(\CAlg^{\cn}_R, \widehat{\SSet}) \rightarrow \Fun(\CAlg^{\cn}, \widehat{\SSet})_{/R}$ of \ref{quasi-geometric spectral algebraic stacks over any base}, the canonical map $\QStk^{\PSt}(X)\rightarrow \QStk^{\PSt}_R(X')$ is an equivalence of $\infty$-categories by construction. 
\end{pg}

\begin{remark}\label{quasi-coherent stacks in DAG}
	Let ${\QStk'}^{\PSt}_R: \Fun(\CAlg^{\Delta}_R, \widehat{\SSet})^{\op} \rightarrow \widehat{\Cat}_\infty$ denote the functor obtained by applying \cite[6.2.1.11]{lurie2018sag} to the projection $q'_R:\CAlg^{\Delta}_R\times_{\CAlg^{\cn}} \LinCat^{\PSt}\rightarrow \CAlg^{\Delta}_R$. By construction, the canonical map $\QStk^{\PSt}_R\circ \Theta_{R!}^{\op} \rightarrow {\QStk'}^{\PSt}_R$ is an equivalence of functors. According to \cite[D.4.1.6]{lurie2018sag}, the functor $\chi_R: \CAlg^{\cn}_R\rightarrow \widehat{\Cat}_\infty$ classifying the coCartesian fibration $q_R$ is a sheaf for the flat universal descent topology of \cite[D.4.1.4]{lurie2018sag}, and therefore also a sheaf for the fiber smooth topology of \ref{fiber smooth topology} by virtue of \cite[11.2.3.3]{lurie2018sag}. Using \cite[1.3.1.7]{lurie2018sag}, we deduce that the canonical map $\QStk^{\PSt}_R\circ L_{\fsm}^{\op} \rightarrow \QStk^{\PSt}_R$ is an equivalence of functors from $\Fun(\CAlg^{\cn}_R, \widehat{\SSet})^{\op}$ to $\widehat{\Cat}_\infty$. We conclude that for each quasi-geometric derived algebraic stack $X$ over $R$, the canonical map $\QStk^{\PSt}_R(X^\circ) \rightarrow {\QStk'}^{\PSt}_R(X)$ is an equivalence of $\infty$-categories. 
\end{remark}

\begin{remark}\label{the necessity of fiber smooth topology}
	Our choice of the fiber smooth topology over the fpqc topology for the underlying quasi-geometric spectral algebraic stacks of quasi-geometric derived algebraic stacks (see \ref{quasi-geometric spectral algebraic stacks from DAG}) is motivated by \ref{quasi-coherent stacks in DAG}. Indeed, we do not know if the functor $\chi_R$ satisfies fpqc descent. 
\end{remark}

\begin{pg}\label{global section functor for quasi-geometric stacks}
	Let $X$ be a quasi-geometric stack which satisfies the following condition: 
\begin{itemize}
\item[$(\ast)$] There exists a morphism of quasi-geometric stacks $\sfX_0 \rightarrow X$ which is faithfully flat and locally almost of finite presentation, where $\sfX_0$ is affine.
\end{itemize}
Let $\Groth_\infty$ denote the $\infty$-category of Grothendieck prestable $\infty$-categories (see \cite[C.3.0.5]{lurie2018sag}). Let $S$ denote the sphere spectrum and let $q: X \rightarrow \Spec S$ be the projection. It follows from \cite[10.4.1.1]{lurie2018sag} that the pullback functor $q^\ast: \Groth_\infty \rightarrow \QStk^{\PSt}(X)$ induced by $q$ (note that there is a canonical equivalence $\QStk^{\PSt}(\Spec S)\simeq  \Groth_\infty$) admits a right adjoint 
$$
\QCoh(X; \bullet):\QStk^{\PSt}(X) \rightarrow \Groth_\infty
$$ 
which we refer to as the \emph{global section functor on $X$}. For each prestable quasi-coherent stack $\calC$ on $X$, we refer to $\QCoh(X; \calC)$ as the \emph{$\infty$-category of global sections} of $\calC$.
\end{pg}

\begin{remark}\label{a description of global sections of quasi-coherent stacks}
	Let $f_0:Y_0 \rightarrow X$ be a morphism between quasi-geometric stacks satisfying condition $(\ast)$ of \ref{global section functor for quasi-geometric stacks}. Assume that $f_0$ is representable faithfully flat and locally almost of finite presentation. Applying the argument of \cite[10.4.1.4]{lurie2018sag} to $f_0$, we deduce that the canonical morphism $\QCoh(X; \calC)\rightarrow \lim\limits_{[n]\in \Delta} \QCoh(Y_n; f_n^\ast \calC)$ is an equivalence in $\Groth_\infty$ (here $Y_\bullet$ denotes the \Cech nerve of $f$ and each $f_n:Y_n \rightarrow X$ denotes the projection). 
\end{remark}

\begin{remark}\label{stability of global section}
	Let $\Delta_{s}$ denote the subcategory of $\Delta$ having the same objects but the morphisms are given by \emph{injective} order-preserving maps (see \cite[6.5.3.6]{MR2522659}). Let $\Groth_\infty^{\lex}$ denote the subcategory of $\widehat{\Cat}_\infty$ whose objects are Grothendieck prestable $\infty$-categories and whose morphisms are functors preserving small colimits and finite limits; see \cite[C.3.2.3]{lurie2018sag}. By virtue of \cite[C.3.2.4]{lurie2018sag}, $\Groth_\infty^{\lex}$ admits small limits and the inclusion $\Groth_\infty^{\lex}\subseteq \Groth_\infty$ preserves small limits. In the situation of \ref{a description of global sections of quasi-coherent stacks}, the existence of the limit $\lim\limits_{[n]\in \Delta} \QCoh(Y_n; f_n^\ast \calC)$ is supplied by the right cofinality of the inclusion $\Delta_{s}\subseteq \Delta$ (see \cite[6.5.3.7]{MR2522659}) and our assumption that $f_0$ is flat. Indeed, the limit is given by $\lim\limits_{[n]\in \Delta_{s}} \QCoh(Y_n; f_n^\ast \calC)$, where the flatness assumption guarantees that the construction $[n] \mapsto \QCoh(Y_n; f_n^\ast \calC)$ determines a functor $\Delta_{s}\rightarrow \Groth_\infty^{\lex}$ (cf. \cite[10.1.7.10]{lurie2018sag}). Let $\LPres$ denote the subcategory of $\widehat{\Cat}_\infty$ whose objects are presentable $\infty$-categories and whose morphisms are functors which preserve small colimits; see \cite[5.5.3.1]{MR2522659}. In the special case where $\calC$ is stable, combining \cite[10.3.1.8]{lurie2018sag} with \cite[4.8.2.18]{lurie2017ha} (which asserts that the full subcategory $\Pr^{\St}\subseteq \LPres$ spanned by the presentable stable $\infty$-categories is closed under small limits), we deduce that $\QCoh(X; \calC)$ is stable.
\end{remark}

\begin{pg}\label{the full subcategory spanned by objects with prescribed support}
	Let $f:X\rightarrow Y$ be a representable morphism between quasi-geometric stacks satisfying condition $(\ast)$ of \ref{global section functor for quasi-geometric stacks}. It follows from \cite[10.1.4.1]{lurie2018sag} that the pullback functor $f^\ast: \QStk^{\PSt}(Y)\rightarrow  \QStk^{\PSt}(X)$ admits a right adjoint $f_\ast$. Let $\calC$ be a prestable quasi-coherent stack on $Y$. Applying the global section functor on $Y$ (see \ref{global section functor for quasi-geometric stacks}) to the unit morphism $\calC\rightarrow f_\ast f^\ast \calC$, we obtain a functor $\QCoh(Y;\calC)\rightarrow \QCoh(X; f^\ast \calC)$ which we refer to as the \emph{pullback along $f$} and denote by $f^\ast$ (see \cite[10.1.7.5]{lurie2018sag}). 

Let $j:U\rightarrow X$ be an open immersion of quasi-geometric stacks satisfying condition $(\ast)$ of \ref{global section functor for quasi-geometric stacks}. For each object $\calC\in \QStk^{\PSt}(X)$, we let $\QCoh_{X-U}(X; \calC)\subseteq \QCoh(X; \calC)$ denote the full subcategory spanned by those objects $M$ such that $j^\ast M \in \QCoh(U; j^\ast \calC)$ is equivalent to $0$. 
\end{pg}

\begin{pg}\label{the unit object of QStk}
	Let $X$ be a quasi-geometric stack satisfying condition $(\ast)$ of \ref{global section functor for quasi-geometric stacks}. Let $\QStk^{\St}(X) \subseteq \QStk^{\PSt}(X)$ denote the full subcategory spanned by the stable quasi-coherent stacks (see \cite[10.1.2.1]{lurie2018sag}). Let $\calQ_X$ denote the unit object of $\QStk^{\St}(X)$ (with respect to the symmetric monoidal structure described in \cite[10.1.6.4]{lurie2018sag}). More informally, it assigns to each point $\eta\in X(R)$ the stable $R$-linear $\infty$-category $\Mod_R$. Note that $\calQ_X$ is compactly generated (see \cite[7.2.4.2]{lurie2017ha}) and that there is a canonical equivalence of $\infty$-categories $\QCoh(X; \calQ_X)\simeq \QCoh(X)$. More generally, let $j:U\rightarrow X$ be a representable open immersion. Repeating the argument of \cite[10.1.7.3]{lurie2018sag}, we obtain a stable quasi-coherent stack $\calQ_{X-U}$ on $X$, which is given informally by the formula $(\calQ_{X-U})_\eta=\Mod_R^{\mathrm{Nil}(I_\eta)}$ for each point $\eta\in X(R)$ (here $I_\eta\subseteq \pi_0R$ is a finitely generated ideal whose vanishing locus is complementary to the open subset $|\Spec R\times_XU| \subseteq |\Spec R|$ and $\Mod_R^{\mathrm{Nil}(I_\eta)}\subseteq \Mod_R$ denotes the full subcategory spanned by the $I_\eta$-nilpotent objects of \cite[7.1.1.6]{lurie2018sag}). For each stable quasi-coherent stack $\calC$ on $X$, let $\calC_{X-U}$ denote the tensor product $\calC\otimes \calQ_{X-U}$ in the symmetric monoidal $\infty$-category $\QStk^{\St}(X)$. For an alternative description, let $\eta\in X(R)$ be a point and let $(\calC_\eta)^{\mathrm{Nil}(I_\eta)}\subseteq \calC_\eta$ denote the full subcategory spanned by the $I_\eta$-nilpotent objects. It then follows from \cite[7.1.2.11]{lurie2018sag} (and its proof) that $\calC_{X-U}$ is equivalent to the stable quasi-coherent stack determined by the construction $(\eta\in X(R))\mapsto (\calC_\eta)^{\mathrm{Nil}(I_\eta)}$; in particular, \cite[7.1.1.12]{lurie2018sag} guarantees that if $\calC$ is compactly generated, then so is $\calC_{X-U}$. We have the following observation:
\end{pg}

\begin{lemma}\label{quasi-coherent stacks and support}
	Let $j:U\rightarrow X$ be a representable open immersion of quasi-geometric stacks satisfying condition $(\ast)$ of \emph{\ref{global section functor for quasi-geometric stacks}}. Let $\calC$ be a stable quasi-coherent stack on $X$ and let $\calC_{X-U}$ be as in \emph{\ref{the unit object of QStk}}. Then the canonical morphism $\QCoh(X; \calC_{X-U})\rightarrow \QCoh(X; \calC)$ induces an equivalence of $\infty$-categories $\QCoh(X; \calC_{X-U})\rightarrow \QCoh_{X-U}(X; \calC)$ (see \emph{\ref{the full subcategory spanned by objects with prescribed support}}).
\end{lemma}
\begin{proof}
	By virtue of \ref{a description of global sections of quasi-coherent stacks}, we are reduced to the case where $X$ is a quasi-geometric spectral Deligne-Mumford stack. Using the proof of \cite[10.1.4.1]{lurie2018sag}, we can reduce further to the case where $X$ is affine, in which case the desired result follows immediately from the definition of the $\infty$-category $\calC^{\mathrm{Nil}(I)}$ appearing in \cite[7.1.1.6]{lurie2018sag}.
\end{proof}

\begin{pg}
	The following pair of results asserts that if a quasi-geometric spectral algebraic stack is of twisted compact generation, then it satisfies a spectral analogue of the $\beta$-Thomason condition of \cite[8.1]{MR3705292} for some regular cardinal $\beta$ (see \ref{beta-Thomason condition}):
\end{pg}

\begin{lemma}\label{twisted compact generation and global sections}
	Let $X$ be a quasi-geometric spectral algebraic stack which is of twisted compact generation. Then for each representable open immersion $j:U\rightarrow X$ of quasi-geometric spectral algebraic stacks, the $\infty$-category $\QCoh_{X-U}(X)$ is compactly generated. In particular, $\QCoh(X)$ is compactly generated. 
\end{lemma}
\begin{proof}
	It follows from \cite[7.1.1.12]{lurie2018sag} that $\calQ_{X-U}$ is compactly generated and stable, so the desired result follows immediately by applying \ref{quasi-coherent stacks and support} to $\calC=\calQ_X$ 
\end{proof}

\begin{pg}
	Before stating our next result, we introduce some terminology. Let $X:\CAlg^{\cn}\rightarrow \widehat{\SSet}$ be a functor which satisfies condition $(\ast)$ of \ref{the induced topological space of quasi-geometric stacks}. Suppose we are given a perfect object $F$ of $\QCoh(X)$. Let $\Supp F$ denote the subset of $|X|$ consisting of those elements $x\in |X|$ such that for any point $\eta: \Spec \kappa \rightarrow X$ representing $x$, $\eta^\ast F \neq 0$. This set is well-defined; we refer to it as the \emph{support} of $F$. In the special case where $X$ is representable by a quasi-separated spectral algebraic space, it follows from \cite[7.1.5.5]{lurie2018sag} that our definition of support is compatible with the definition of support in the sense of \cite[7.1.5.4]{lurie2018sag}. 
\end{pg}

\begin{lemma}\label{compact objects with prescribed support}
	Let $X$ be a quasi-geometric spectral algebraic stack which is of twisted compact generation. Then for each quasi-compact open subset $\calU\subseteq |X|$, there exists a compact object $F$ of $\QCoh(X)$ with support $|X|-\calU$. 
\end{lemma}
\begin{proof}
	\ref{quasi-compact open subsets and quasi-geometric stacks} shows that there exist a quasi-geometric spectral algebraic stack $U$ and a representable open immersion $j:U \rightarrow X$ such that $|j|(|U|)=\calU$. Then \ref{twisted compact generation and global sections} supplies a set of compact generators $\{F_i\}_{i\in I}$ for $\QCoh_{X-U}(X)$. Using \cite[6.3.4.1]{lurie2018sag}, we observe that $j^\ast:\QCoh(X)\rightarrow \QCoh(U)$ admits a fully faithful right adjoint $j_\ast$. It then follows from \cite[7.2.1.7]{lurie2018sag} that the pair of subcategories $(\QCoh_{X-U}(X), \QCoh(U))$ determine a semi-orthogonal decomposition of $\QCoh(X)$ (see \cite[7.2.0.1]{lurie2018sag}). Using \cite[7.2.1.4]{lurie2018sag}, we see that the inclusion $\QCoh_{X-U}(X)\subseteq \QCoh(X)$ admits a right adjoint which we denote by $R$. According to \cite[7.2.0.2]{lurie2018sag}, there is a fiber sequence $R(P) \rightarrow P \rightarrow j_\ast j^\ast P$ for each $P \in \QCoh(X)$. Since $j_\ast$ preserves small colimits \cite[6.3.4.3]{lurie2018sag} (and $\QCoh(X)$ is stable), $R$ preserves filtered colimits, and therefore each $F_i$ is also compact as an object of $\QCoh(X)$ (see \cite[5.5.7.2]{MR2522659}). In particular, every $F_i\in \QCoh(X)$ is perfect by virtue of \cite[9.1.5.2]{lurie2018sag}. We now proceed as in the proof of \cite[4.10]{MR3705292}. We first claim that the $|X|-\calU$ can be identified with the union of the supports $\{\Supp F_i\}$. For this, it suffices to prove that $|X|-\calU$ belongs to the union of $\{\Supp F_i\}$. Let $x\in |X|-\calU$ and choose a point $\eta: \Spec \kappa \rightarrow X$ representing $x$. Then \cite[6.3.4.1]{lurie2018sag} guarantees that $\eta_\ast \calO_{\Spec \kappa} \in \QCoh_{X-U}(X)$. Since $\{F_i\}$ is a set compact generators for $\QCoh_{X-U}(X)$ and $\eta_\ast \calO_{\Spec \kappa}$ is nonzero, we conclude that $\eta^\ast F_i$ is not equivalent to $0$ for some $i$, so that $x \in \Supp F_i$ as desired. To complete the proof, it will suffice to show that $|X|-\calU$ can be covered by finitely many supports $\Supp F_i$. Choose a fiber smooth surjection $f:\sfX_0\rightarrow X$, where $\sfX_0$ is affine. Replacing $X$ by $\sfX_0$, we are reduced to the case where $X$ is affine. In this case, \cite[3.6.3.4]{lurie2018sag} guarantees that $|X|$ is a coherent topological space (see also \ref{compatible underlying topological spaces}). Using \cite[7.1.5.5]{lurie2018sag}, we see that each $\Supp F_i$ is a closed subset of $|X|$ complementary to a quasi-compact open subset. Consequently, every $\Supp F_i$ and $|X|-\calU$ are constructible, so that the desired result follows by using the constructible topology on $|X|$ of \cite[4.3.1.5]{lurie2018sag}.
\end{proof}

\begin{pg}
	It is a tautology that the class of quasi-geometric spectral algebraic stacks which are of twisted compact generation includes the basic building blocks of spectral algebraic geometry: 
\end{pg}

\begin{lemma}\label{affines are of twisted compact generation}
	Let $\sfX$ be an affine spectral Deligne-Mumford stack. Then $\sfX$ is of twisted compact generation. 
\end{lemma}

\begin{pg}
	Our goal in this section is to prove that quasi-geometric spectral algebraic stacks which admit a quasi-finite presentation are of twisted compact generation \ref{twisted compact generation of spectral algebraic stacks}. We begin by establishing some preliminaries.
\end{pg}

\begin{lemma}\label{pullback square and right adjointability}
	Suppose we are given a pullback diagram of quasi-geometric stacks satisfying condition $(\ast)$ of \emph{\ref{global section functor for quasi-geometric stacks}}:
$$
\Pull{X'}{Y'}{X}{Y.}{f'}{g'}{g}{f}
$$
Let $\calC$ be a prestable quasi-coherent stack on $Y$. If $f$ is a relative spectral algebraic space which is quasi-compact quasi-separated and $g$ is representable flat, then the commutative diagram of $\infty$-categories
$$
\Pull{\QCoh(Y;\calC)}{\QCoh(X;f^\ast \calC)}{\QCoh(Y';g^\ast\calC)}{\QCoh(X';{g'}^\ast f^\ast \calC)}{f^\ast}{g^\ast}{{g'}^\ast}{{f'}^\ast}
$$
is right adjointable. 
\end{lemma}
\begin{proof}
	This follows by combining \cite[10.1.7.9]{lurie2018sag} with an extension of \cite[10.1.7.13]{lurie2018sag} to quasi-geometric stacks satisfying condition $(\ast)$ of \ref{global section functor for quasi-geometric stacks} (which can be achieved by using \cite[4.7.4.18]{lurie2017ha}). 
\end{proof}

\begin{lemma}\label{compact pullback functor}
	Let $f:X\rightarrow Y$ be a relative spectral algebraic space which is quasi-compact quasi-separated morphism of quasi-geometric stacks satisfying condition $(\ast)$ of \emph{\ref{global section functor for quasi-geometric stacks}}. Let $\calC$ be a prestable quasi-coherent stack on $Y$. Then the pullback functor $f^\ast: \QCoh(Y;\calC) \rightarrow \QCoh(X;f^\ast \calC)$ is compact. 
\end{lemma}
\begin{proof}
	By virtue of \ref{a description of global sections of quasi-coherent stacks} and \ref{stability of global section}, we can reduce to the case where $X$ and $Y$ are representable by spectral Deligne-Mumford stacks, in which case the desired result follows by combining \cite[10.1.7.15]{lurie2018sag} with \cite[10.3.1.13]{lurie2018sag}.
\end{proof}

\begin{pg}
	As a first step towards the proof of \ref{twisted compact generation of spectral algebraic stacks}, we show that the property of being of twisted compact generation behaves well with respect to open immersions:
\end{pg}

\begin{proposition}\label{open immersion and twisted compact generation}
	Let $j:U\rightarrow X$ be an open immersion of quasi-geometric spectral algebraic stacks. If $X$ is of twisted compact generation, then so is $U$.
\end{proposition}
\begin{proof}
	Let $\calC$ be a compactly generated stable quasi-coherent stack on $U$; we wish to show that $\QCoh(U; \calC)$ is compactly generated. It follows from \ref{pullback square and right adjointability} and \ref{compact pullback functor} that we have an adjunction
$$
\Adjoint{j^\ast}{\QCoh(X; j_\ast \calC)}{\QCoh(U; \calC)}{j_\ast,}
$$
where $j_\ast$ is conservative and preserves small filtered colimits. Using \cite[10.3.2.3]{lurie2018sag} and \cite[10.3.1.7]{lurie2018sag} (see also \cite[10.1.4.1]{lurie2018sag}), we see that $j_\ast \calC \in \QStk^{\PSt}(X)$ is compactly generated and stable. Since $X$ is of twisted compact generation, the desired result follows from \cite[6.2]{lurie2011dagxi}.
\end{proof}

\begin{pg}
	Our proof of \ref{twisted compact generation of spectral algebraic stacks} will require two descent results about the property of being of twisted compact generation. Let us begin with the descent along finite morphisms. In what follows, we regard $\LPres$ as equipped with the symmetric monoidal structure described in \cite[4.8.1.15]{lurie2017ha}. Note that $\Pres^{\St}$ inherits a symmetric monoidal structure from $\LPres$ (see \cite[4.8.2.18]{lurie2017ha}). For the first descent result, we need the following stable version of \cite[10.2.4.2]{lurie2018sag}:
\end{pg}

\begin{lemma}\label{quasi-affine base change in the stable case}
	Let $f:X\rightarrow Y$ be a representable quasi-affine morphism of quasi-geometric stacks satisfying condition $(\ast)$ of \emph{\ref{global section functor for quasi-geometric stacks}}, and let $\calC$ be a stable quasi-coherent stack on $Y$. Then the canonical morphism
$$
\QCoh(X) \otimes_{\QCoh(Y)}\QCoh(Y; \calC) \rightarrow \QCoh(X; f^\ast \calC)
$$
is an equivalence of presentable stable $\infty$-categories.
\end{lemma}
\begin{proof}
	We first claim that the functor $\Mod_{\QCoh(Y)}(\Pres^{\St}) \rightarrow \Pres^{\St}$ which carries a presentable stable $\infty$-category $\calD$ equipped with an action of $\QCoh(Y)$ to the presentable stable $\infty$-category $\QCoh(X)\otimes_{\QCoh(Y)}\calD$ preserves small limits. For this, it will suffice to show that $\QCoh(X)$ is dualizable when viewed as an object of $\Mod_{\QCoh(Y)}(\Pres^{\St})$. Let $\calA \in \CAlg(\QCoh(Y))$ denote the pushforward of the structure sheaf of $X$ (that is, the unit object of $\QCoh(X)$) along $f$. Then \cite[6.3.4.6]{lurie2018sag} supplies an equivalence $\QCoh(X)\rightarrow \Mod_{\calA}(\QCoh(Y))$, so the desired assertion follows from \cite[4.8.4.8]{lurie2017ha}. Combining this observation with \ref{a description of global sections of quasi-coherent stacks} (see also \cite[6.3.4.7]{lurie2018sag}), we are reduced to the case where $X$ and $Y$ are representable by spectral Deligne-Mumford stacks $\sfX$ and $\sfY$, respectively. In this case, it follows from \cite[10.2.1.1]{lurie2018sag} that $\mathrm{LMod}_{\calA}(\QCoh(\sfY; \calC))$ can be identified with the stabilization of $\mathrm{LMod}_{\calA}(\QCoh(\sfY; \calC))_{\geq 0}$ (see \cite[10.2.1.1]{lurie2018sag} for the t-structure). By virtue of the equivalence $\QCoh(\sfX)\simeq \Mod_{\calA}(\QCoh(\sfY))$ and \cite[4.8.4.6]{lurie2017ha}, it can also be identified with $\QCoh(\sfX)\otimes_{\QCoh(\sfY)} \QCoh(\sfY; \calC)$. Invoking our assumption that $f$ is quasi-affine, we obtain an equivalence $\mathrm{LMod}_{\calA}(\QCoh(\sfY;\calC))_{\geq 0}\simeq \QCoh(\sfX, f^\ast  \calC)$ by combining \cite[10.2.1.3]{lurie2018sag} and \cite[10.2.4.2]{lurie2018sag}; the desired equivalence now follows by passing to the stabilization.
\end{proof}

\begin{proposition}\label{finite morphisms and twisted compact generation}
	Let $f:X\rightarrow Y$ be a morphism of quasi-geometric spectral algebraic stacks which is representable, finite, faithfully flat, and locally almost of finite presentation. If $X$ is of twisted compact generation, then so is $Y$.
\end{proposition}
\begin{proof}
	Let $\calC$ be a compactly generated stable quasi-coherent stack on $Y$; we wish to show that $\QCoh(Y; \calC)$ is compactly generated. We first show that the pullback functor $f^\ast: \QCoh(Y; \calC) \rightarrow \QCoh(X; f^\ast \calC)$ admits a left adjoint. Since it preserves small colimits (see \ref{pullback square and right adjointability}), it will suffice to prove that it preserves small limits by virtue of the adjoint functor theorem \cite[5.5.2.9]{MR2522659}. Let $\calA \in \CAlg(\QCoh(Y))$ denote the pushforward $f_\ast \calO_X$. Combining \ref{quasi-affine base change in the stable case} with \cite[6.3.4.6]{lurie2018sag} and \cite[4.8.4.6]{lurie2017ha}, we can identify $\QCoh(X; f^\ast \calC)$ with $\mathrm{LMod}_{\calA}(\QCoh(Y; \calC))$, under which $f^\ast$ corresponds to the functor $\QCoh(Y; \calC) \rightarrow \mathrm{LMod}_{\calA}(\QCoh(Y; \calC))$
given by tensor product with $\calA$. Since the forgetful functor $\mathrm{LMod}_{\calA}(\QCoh(Y; \calC))\rightarrow \QCoh(Y; \calC)$ is conservative \cite[4.2.3.2]{lurie2017ha} and preserves small limits \cite[4.2.3.3]{lurie2017ha}, it is enough to prove that $\calA$ is dualizable as an object of $\QCoh(Y)$. Invoking our assumption on $f$ (and using \cite[6.3.4.1]{lurie2018sag}), we deduce from \cite[6.1.3.2]{lurie2018sag} that $\calA$ is perfect, hence dualizable as desired (see \cite[6.2.6.2]{lurie2018sag}). Now let $f_+$ denote a left adjoint to $f^\ast$. Since $X$ is of twisted compact generation and $f^\ast$ is conservative (by virtue of \ref{a description of global sections of quasi-coherent stacks}), the desired compact generation follows immediately by applying \cite[6.2]{lurie2011dagxi} to the adjoint pair $(f_+, f^\ast)$. 
\end{proof}

\begin{pg}
	We next show that the property of being a twisted compact generation satisfies descent for excision squares. For this purpose, the Thomason--Neeman localization theorem in the setting of $\infty$-categories by Adeel A. Khan (see \cite[2.11]{khanlecture3}) will play an essential role. Of greatest interest to us is the case of semi-orthogonal decompositions \cite[7.2.0.1]{lurie2018sag}:
\end{pg}

\begin{lemma}\label{a special case of [2.11]{khanlecture3}}
	Let $\calC$ be a compactly generated presentable stable $\infty$-category and let $(\calC_+,\calC_-)$ be a semi-orthogonal decomposition of $\calC$ for which the subcategories $\calC_+, \calC_-$ are compactly generated. Suppose that the inclusion functor $\iota:\calC_-\rightarrow \calC$ preserves small filtered colimits. Let $L$ denote a left adjoint to the inclusion $\iota$ (which exists by virtue of \emph{\cite[7.2.1.7]{lurie2018sag}}). Let $D$ be a compact object of $\calC_-$. Then there exists a compact object $C$ of $\calC$ and an equivalence $L(C)\simeq D\oplus D[1]$ in the $\infty$-category $\calC_-$. 
\end{lemma}
\begin{proof}
	By virtue of \cite[7.2.1.4]{lurie2018sag}, the inclusion functor $\calC_+\rightarrow \calC$ admits a right adjoint, which we denote by $R$. According to \cite[7.2.0.2]{lurie2018sag}, for each object $C\in \calC$, there is a fiber sequence $R(C)\rightarrow C \rightarrow L(C)$. Combing this with our assumption on $\iota$, we see that $R$ preserves small filtered colimits. The desired result now follows immediately from \cite[2.11]{khanlecture3}. 
\end{proof}

\begin{pg}
	The following ``Nisnevich excision" result will be useful in the proof of \ref{excision squares and twisted compact generation}:
\end{pg}

\begin{lemma}\label{Nisnevich excision of global section functor}
	Suppose we are given an excision square of quasi-geometric spectral algebraic stacks: 
$$
\Pull{U'}{X'}{U}{X.}{j'}{f'}{f}{j}
$$
Let $\calC$ be a prestable quasi-coherent stack on $X$. If $f$ is quasi-affine, then the commutative diagram of $\infty$-categories 
$$
\Pull{\QCoh(X; \calC)}{\QCoh(U; j^\ast \calC)}{\QCoh(X'; f^\ast \calC)}{\QCoh(U'; {j'}^\ast f^\ast \calC)}{}{}{}{}
$$
is a pullback square in $\Groth_\infty$.
\end{lemma}
\begin{proof}
	We can use \ref{a description of global sections of quasi-coherent stacks} to reduce to the case where $X$ and $Y$ are quasi-affine spectral Deligne-Mumford stacks. In this case, the desired result follows from \cite[10.2.3.1]{lurie2018sag} and \cite[10.2.4.2]{lurie2018sag}.
\end{proof}

\begin{proposition}\label{excision squares and twisted compact generation}
	Suppose we are given an excision square of quasi-geometric spectral algebraic stacks
$$
\Pull{U'}{X'}{U}{X,}{j'}{f'}{f}{j}
$$
where $f$ is quasi-affine. If $X'$ and $U$ are of twisted compact generation, then so is $X$. 
\end{proposition}
\begin{proof}
	Let $\calC \in \QStk^{\PSt}(X)$ be compactly generated stable; we wish to show that $\QCoh(X; \calC)$ is compactly generated. For this, we proceed as in the proofs of \cite[6.20]{lurie2011dagxi} and \cite[6.13]{MR3190610}. We first claim that the collection of objects of the form $j^\ast M$, where $M$ is a compact object of $\QCoh(X; \calC)$, is a set of compact generators for $\QCoh(U; j^\ast \calC)$ (note that $j^\ast M$ is compact by virtue of \ref{compact pullback functor}). Using the compact generation of $\QCoh(U; j^\ast \calC)$ (because $U$ is of twisted compact generation), it suffices to prove that for each compact object $N \in \QCoh(U; j^\ast \calC)$, there exists a compact object $M\in \QCoh(X; \calC)$ such that $j^\ast M \simeq N\oplus N[1]$. By virtue of \ref{compact pullback functor} and \ref{Nisnevich excision of global section functor}, we observe that an object $M\in \QCoh(X; \calC)$ is compact if and only if both $j^\ast M$ and $f^\ast M$ are compact; consequently, in order to lift $N\oplus N[1]$ to a compact object of $\QCoh(X; \calC)$, it suffices to lift ${f'}^\ast(N\oplus N[1])$ to a compact object of $\QCoh(X'; f^\ast \calC)$. Since $j'^\ast$ admits a fully faithful right adjoint $j'_\ast$ \ref{pullback square and right adjointability}, it follows from \cite[7.2.1.7]{lurie2018sag} that the pair of subcategories $(\QCoh_{X'-U'}(X'; f^\ast \calC), \QCoh(U'; {j'}^\ast f^\ast \calC))$ determine a semi-orthogonal decomposition of $\QCoh(X'; f^\ast \calC)$. Combining the assumption that $X'$ is of twisted compact generation with \ref{quasi-coherent stacks and support} and \ref{open immersion and twisted compact generation}, we deduce that $\QCoh(X'; f^\ast \calC)$ and the two subcategories are compactly generated. Since $j'_\ast$ preserves small filtered colimits \ref{compact pullback functor}, the desired lifting follows from the Thomason--Neeman localization theorem \cite[2.11]{khanlecture3} (see \ref{a special case of [2.11]{khanlecture3}}).

On the other hand, \ref{Nisnevich excision of global section functor} guarantees that the pullback functor $f^\ast$ induces an equivalence of $\infty$-categories $\QCoh_{X-U}(X; \calC) \rightarrow \QCoh_{X'-U'}(X'; f^\ast \calC)$, so that $\QCoh_{X-U}(X; \calC)$ is also compactly generated. Let $\{M'_\alpha \}_{\alpha \in A}$ be the collection of compact objects of $\QCoh_{X-U}(X; \calC)$. As in the case of $\QCoh(X'; f^\ast \calC)$, the pair $(\QCoh_{X-U}(X; \calC), \QCoh(U; j^\ast \calC))$ is a semi-orthogonal decomposition of $\QCoh(X; \calC)$. Since $j_\ast$ preserves small filtered colimits, it follows from the proof of \ref{a special case of [2.11]{khanlecture3}} that the inclusion $\QCoh_{X-U}(X; \calC)\subseteq \QCoh(X; \calC)$ admits a right adjoint $R$ which preserves small filtered colimits, and therefore each $M'_\alpha \in \QCoh_{X-U}(X; \calC)$ is compact as an object of $\QCoh(X; \calC)$ as well. To complete the proof, we will show that the collection of objects of the form $M\oplus M'_\alpha \in \QCoh(X; \calC)$, where $M \in \QCoh(X; \calC)$ is compact, is a set of compact generators for $\QCoh(X; \calC)$: let $M'' \in \QCoh(X; \calC)$ and suppose that $\Ext^\ast_{\QCoh(X; \calC)}(M\oplus M'_\alpha, M'')=0$ for each compact object $M\in \QCoh(X;\calC)$ and each $M'_\alpha$. Since the collection $\{M'_\alpha \}$ is a set of compact generators for $\QCoh_{X-U}(X; \calC)$, we see that $RM'' \simeq 0$. It then follows from the fiber sequence $RM''\rightarrow M'' \rightarrow j_\ast j^\ast M''$ (see \cite[7.2.0.2]{lurie2018sag}) that the unit map $M''\rightarrow j_\ast j^\ast M''$ is an equivalence. Combining this with the above discussion of the collection $\{j^\ast M\}$, we deduce that $j^\ast M'' \simeq 0$, hence $M''\simeq 0$ as desired.
\end{proof}

\begin{pg}
	We now provide a proof of our main result of this section. Our basic strategy is analogous to the proof of \cite[Theorem A]{MR3705292}. 
\end{pg}

\begin{proof}[Proof of \emph{\ref{twisted compact generation of spectral algebraic stacks}}]
	In view of the special presentation of \ref{a presentation for applying quasi-finite devissage} and the stacky scallop decomposition induced by $p$ (see \ref{stacky scallop decomposition induced by a Nisnevich covering}), we can use the descent results \ref{finite morphisms and twisted compact generation} and \ref{excision squares and twisted compact generation} (note that each morphism $W_m\rightarrow U_m$ appearing in the proof of \ref{stacky scallop decomposition induced by a Nisnevich covering} is quasi-affine by virtue of the separatedness of $p$ and \cite[3.3.0.2]{lurie2018sag}) along with \ref{open immersion and twisted compact generation} to reduce to the case of quasi-affine spectral Deligne-Mumford stacks, in which case the desired result follows from \ref{affines are of twisted compact generation} and \ref{open immersion and twisted compact generation} (see also \cite[2.4.2.3]{lurie2018sag}).
\end{proof}

\begin{pg}
	We close this section by mentioning another consequence of the special presentation of \ref{a presentation for applying quasi-finite devissage} and the descent results \ref{finite morphisms and twisted compact generation} and \ref{excision squares and twisted compact generation}. According to \cite[D.5.3.1]{lurie2018sag}, the property of being a compactly generated prestable $R$-linear $\infty$-category is local for the \'etale topology, where $R$ is a connective $\bbE_\infty$-ring. In the stable case, we have the following stronger assertion, whose proof is immediate from the proofs of \ref{finite morphisms and twisted compact generation} and \ref{excision squares and twisted compact generation} by considering \ref{a presentation for applying quasi-finite devissage} and \ref{stacky scallop decomposition induced by a Nisnevich covering}:
\end{pg}

\begin{theorem}\label{compact generation is local for the quasi-finite topology}
	Let $X$ be a quasi-geometric spectral algebraic stack and let $\calC$ be a stable quasi-coherent stack on $X$. If there exists a locally quasi-finite, faithfully flat, and locally almost of finite presentation of quasi-geometric spectral algebraic stacks $f:\Spec A \rightarrow X$ for which $f^\ast \calC$ is compactly generated, then $\calC$ is compactly generated. 
\end{theorem}

\section{Brauer Spaces and Azumaya Algebras}\label{Sec: Brauer Spaces and Azumaya Algebras}
	Our goal in this section is to give a proof of our main result \ref{Brauer spaces and Azumaya algebras for quasi-geometric spectral algebraic stacks}, and to provide an explicit description of the homotopy groups of the extended Brauer sheaf of a quasi-geometric stack. 
	
\begin{pg}
	Let $R$ be a connective $\bbE_\infty$-ring. According to \cite[11.5.3.1]{lurie2018sag}, an $\bbE_1$-algebra $A$ over $R$ is called the \emph{Azumaya algebra over $R$} if it is a compact generator of $\Mod_R$ and the natural map $A\otimes_R A^{\rev} \rightarrow \End_R(A)$ induced by the left and right actions of $A$ on itself is an equivalence. In the setting of derived algebraic geometry, the property of being a derived Azumaya algebra (see \cite[2.1]{MR2957304}) is local for the flat topology \cite[2.3]{MR2957304}. We have the following spectral analogue, which can be proven by exactly the same argument:
\end{pg}

\begin{lemma}\label{Azumaya algebras are local for the flat topology}
	Let $R$ be a connective $\bbE_\infty$-ring and let $A$ be an $\bbE_1$-algebra over $R$. Then the condition that $A$ is an Azumaya algebra is stable under base change \emph{\cite[6.2.5.1]{lurie2018sag}} and local for the flat topology \emph{\cite[2.8.4.1]{lurie2018sag}}.
\end{lemma}

\begin{pg}
	Before giving our proof of the main result \ref{Brauer spaces and Azumaya algebras for quasi-geometric spectral algebraic stacks}, we need to recall a bit of terminology. Let $X: \CAlg^{\cn} \rightarrow \widehat{\SSet}$ be a functor. Let $\QStk^{\cg}(X)$ denote the subcategory of $\QStk^{\St}(X)$ whose objects are compactly generated stable quasi-coherent stacks and whose morphisms are compact morphisms of quasi-coherent stacks of \cite[10.1.3.1]{lurie2018sag}. Note that it inherits a symmetric monoidal structure from $\QStk^{\St}(X)$; see \cite[11.4.0.1]{lurie2018sag}. According to \cite[11.5.2.1]{lurie2018sag}, the \emph{extended Brauer space} $\sBr^\dagger(X)\subseteq \QStk^{\cg}(X)^\simeq$ of $X$ is defined to be the full subcategory spanned by the invertible objects of $\QStk^{\cg}(X)$ (here $\QStk^{\cg}(X)^\simeq$ denotes the largest Kan complex contained in $\QStk^{\cg}(X)$), and the \emph{extended Brauer group} $\Br^\dagger(X)$ of $X$ is defined to be the set $\pi_0\sBr^\dagger(X)$. As in \cite[11.5.2.11]{lurie2018sag}, we denote the full subcategory of $ \QCoh(X)^\simeq$ spanned by the invertible objects of $\QCoh(X)$ by $\sPic^\dagger(X)$ and refer to it as the \emph{extended Picard space} of $X$. Using the symmetric monoidal structures on $\QStk^{\St}(X)$ and $\QCoh(X)$, we may regard $\sBr^\dagger(X)$ and $\sPic^\dagger(X)$ as grouplike commutative monoid objects of the $\infty$-category $\widehat{\SSet}$. 

	According to \cite[11.5.3.7]{lurie2018sag}, an associative algebra object $\calA$ of $\QCoh(X)$ is called the \emph{Azumaya algebra} if, for every morphism $\eta: \Spec R \rightarrow X$ where $R$ is a connective $\bbE_\infty$-ring, $\eta^\ast \calA \in \Alg_R$ is an Azumaya algebra over $R$. For each Azumaya algebra $\calA\in \Alg(\QCoh(X))$, it follows from \cite[11.5.3.9]{lurie2018sag} that the stable quasi-coherent stack on $X$ given by the formula $(\eta:\Spec R \rightarrow X)\mapsto (\RMod_{\eta^\ast \calA}\in \LinCat^{\St}_R)$ determines an object of the extended Brauer space $\sBr^\dagger(X)$. We refer to the equivalence class of this quasi-coherent stack as the \emph{extended Brauer class} of $\calA$ and denote it by $[\calA]\in \sBr^\dagger(X)$.
\end{pg}

\begin{pg}
	We are now ready to prove our main result which extends \cite[11.5.3.10]{lurie2018sag} (which asserts that if $\sfX$ is a quasi-compact quasi-separated spectral algebraic space, then every object of $\Br^\dagger(\sfX)$ has the form $[\calA]$ for some Azumaya algebra $\calA \in \Alg(\QCoh(\sfX))$) to quasi-geometric spectral algebraic stacks which admit a quasi-finite presentation (see \ref{a quasi-finite presentation}). The main ingredient in the proof of \cite[11.5.3.10]{lurie2018sag} is \cite[10.3.2.1]{lurie2018sag} which shows that $\QCoh(\sfX; \calC)$ is compactly generated for each compactly generated prestable quasi-coherent stack $\calC$ on $\sfX$. In our case of interest, we will closely follow the proof of \cite[11.5.3.10]{lurie2018sag}, using \ref{twisted compact generation of spectral algebraic stacks} in place of \cite[10.3.2.1]{lurie2018sag}.
\end{pg}

\begin{proof}[Proof of \emph{\ref{Brauer spaces and Azumaya algebras for quasi-geometric spectral algebraic stacks}}]
	Let $u \in \Br^\dagger(X)$ be an element, and choose an invertible object $\calC$ of $\QStk^{\cg}(X)$ which represents $u$. By virtue of \ref{twisted compact generation of spectral algebraic stacks}, $X$ is of twisted compact generation, so that $\QCoh(X; \calC)$ is compactly generated. Choose a set of compact generators $\{C_i\}_{i\in I}$ for $\QCoh(X; \calC)$. Choose a fiber smooth surjection $f:\Spec A \rightarrow X$, where $A$ is a connective $\bbE_\infty$-ring. Since $f$ is quasi-affine, $f^\ast: \QCoh(X; \calC) \rightarrow \QCoh(\Spec A; f^\ast \calC)$ admits a right adjoint $f_\ast$ (see \ref{pullback square and right adjointability}), and it follows from the proof of \ref{finite morphisms and twisted compact generation} that $\QCoh(\Spec A; f^\ast \calC)$ can be identified with $\mathrm{LMod}_{\calA}(\QCoh(Y; \calC))$, under which the pushforward $f_\ast$ corresponds to the forgetful functor $\mathrm{LMod}_{\calA}(\QCoh(Y; \calC)) \rightarrow \QCoh(Y; \calC)$. In particular, $f_\ast$ is conservative. Combining this observation with the fact that $f^\ast$ is compact (see \ref{compact pullback functor}), we deduce that $\{f^\ast C_i\}_{i\in I}$ is a set of compact generators for $f^\ast \calC$ (see \ref{adjunction and a set of compact generators}). Using \cite[11.5.2.5]{lurie2018sag}, we see that $f^\ast \calC$ is smooth over $A$, and therefore the proof of \cite[11.3.2.4]{lurie2018sag} guarantees that there exists a finite subset $I_0\subset I$ such that the pullback of $C=\oplus_{i\in I_0}C_i$ along $f$ is a compact generator of $f^\ast \calC$. Let $\calE\in \Alg(\QCoh(X))$ denote the endomorphism algebra of $C$ (here we regard $\QCoh(X; \calC)$ as tensored over $\QCoh(X)$), and let $\calC'$ be the stable quasi-coherent stack on $X$, given by the formula $(\eta:\Spec R \rightarrow X)\mapsto (\RMod_{{\eta}^\ast\calE}\in \LinCat^{\St}_R)$. Consider the morphism of quasi-coherent stacks $F:\calC'\rightarrow \calC$ determined by the operation $\bullet \otimes_{\calE}C$. We will complete the proof by showing that the functor $F$ is an equivalence and that $\calE$ is an Azumaya algebra on $X$. By virtue of \ref{Azumaya algebras are local for the flat topology} and \cite[D.4.1.6]{lurie2018sag} (see also \cite[11.2.3.3]{lurie2018sag}), it will suffice to show the assertion after pulling back along $f$. Invoking the fact that $f^\ast C$ is a compact generator of $f^\ast \calC$, we deduce that $f^\ast F$ is an equivalence of $A$-linear $\infty$-categories (see \cite[7.1.2.1]{lurie2017ha}). Combining this observation with the fact that $\calC$ is invertible and \cite[11.5.3.4]{lurie2018sag}, we conclude that $f^\ast\calE$ is an Azumaya algebra over $A$ as desired.
\end{proof}

\begin{pg}
	The remainder of this section is devoted to describing the homotopy groups of the extended Brauer sheaf $\underline{\sBr}^\dagger(X)$ where $X$ is a quasi-geometric stack.
\end{pg}

\begin{definition}
	Let $X$ be a quasi-geometric stack. Let $\CAlg^{\cn}_X\rightarrow \CAlg^{\cn}$ be a left fibration classified by $X$ (so that an object of $\CAlg^{\cn}_X$ can be identified with a pair $(A, \eta)$, where $A$ is a connective $\bbE_\infty$-ring and $\eta\in X(A)$ is an $A$-valued point of $X$). The $\infty$-category $(\CAlg^{\cn}_X)^{\op}$ can be equipped with a Grothendieck topology which we refer to as \emph{the big \'etale topology}: a sieve on an object $(A,\eta)$ is a covering if it contains a finite collection of morphisms $\{(A,\eta)\rightarrow (A_i,\eta_i)\}_{1\leq i\leq n}$ for which the induced map $A\rightarrow \prod A_i$ is faithfully flat and \'etale. There is an induced Grothendieck topology on the opposite of the full subcategory $\CAlg^{\cn, \fpqc}_X \subseteq \CAlg^{\cn}_X$ spanned by those objects $(A, \eta)$ for which the corresponding morphism $\Spec A \rightarrow X$ is flat. We let $\Shv^{\fpqc-\et}_X\subseteq \Fun(\CAlg^{\cn, \fpqc}_X, \SSet)$ denote the full subcategory spanned by the sheaves on $(\CAlg^{\cn, \fpqc}_X)^{\op}$ and refer to it as the \emph{fpqc-\'etale $\infty$-topos} of $X$.
\end{definition}

\begin{remark}
	Let $\underline{\sBr}^\dagger_X, \underline{\sPic}^\dagger_X:\CAlg^{\cn, \fpqc}_X \rightarrow \SSet$ denote the functors given on objects by $(A,\eta)\mapsto \sBr^\dagger(A)$ and $\sPic^\dagger(A)$, respectively. They are fpqc-\'etale sheaves and factor through the $\infty$-category $\CAlg^{\gp}(\SSet)$ of grouplike $\bbE_\infty$-spaces (see \cite[5.2.6.6]{lurie2017ha}). By virtue of \cite[11.5.2.11]{lurie2018sag}, we have a canonical equivalence $\Omega \underline{\sBr}^\dagger_X\simeq \underline{\sPic}^\dagger_X$ in the $\infty$-category of $\CAlg^{\gp}(\SSet)$-valued fpqc-\'etale sheaves on $X$. We refer to $\underline{\sBr}^\dagger_X$ as the \emph{extended Brauer sheaf} of $X$.
\end{remark}

\begin{pg}
	There is an evident forgetful functor $\QCoh(X) \rightarrow \Fun(\CAlg^{\cn}_X, \SSet)$. More informally, it assigns to each $F\in \QCoh(X)$ a functor which carries a pair $(R, \eta) \in \CAlg^{\cn}_X$ to the $0$-th space $\Omega^\infty(F_\eta)$ of the underlying spectrum of the $R$-module $F_\eta$. By virtue of \cite[6.2.3.1]{lurie2018sag}, the forgetful functor factors through the $\infty$-category of big \'etale sheaves on $X$. For any integer $n\geq 0$ and any object $F\in \QCoh(X)$, let $\pi_nF$ denote the $n$-th homotopy group of the restriction of the underlying big \'etale sheaf of $F$ to $\Shv^{\fpqc-\et}_X$. We have the following analogue of \cite[11.5.5.3]{lurie2018sag}, which can be proven by exactly the same argument: 
\end{pg}

\begin{lemma}\label{homotopy sheaves of Brauer sheaf}
	Let $X$ be a quasi-geometric stack. Then the homotopy groups of $\underline{\sBr}^\dagger_X$ are given by
$$
\pi_n\underline{\sBr}^\dagger_X\simeq
\begin{cases}
0& \text{if $n=0$} \\
\underline{\mathbb{Z}} & \text{if $n=1$} \\
(\pi_0\calO_X)^\times & \text{if $n=2$} \\
\pi_{n-2}\calO_X& \text{if $n\geq 3$.} 
\end{cases}
$$
Here $\underline{\mathbb{Z}}$ denotes the constant sheaf associated to the abelian group $\mathbb{Z}$.
\end{lemma}

\begin{pg} 
	In the special case where $X$ is $0$-truncated (in the sense of \cite[9.1.6.1]{lurie2018sag}), the restriction of the underlying big \'etale sheaf of the structure sheaf $\calO_X$ to $\Shv^{\fpqc-\et}_X$ can be regarded as a commutative ring object of the topos of discrete objects of $\Shv^{\fpqc-\et}_X$; let us denote its group of units by $\calO_X^\times$. Arguing as in \cite[11.5.5.4, 11.5.5.5]{lurie2018sag} (using \ref{homotopy sheaves of Brauer sheaf} in place of \cite[11.5.5.3]{lurie2018sag}), \ref{homotopy sheaves of Brauer sheaf} supplies an equivalence $\underline{\sBr}^\dagger_X \simeq K(\calO_X^\times, 2) \times K(\underline{\mathbb{Z}},1)$ in the $\infty$-topos $\Shv^{\fpqc-\et}_X$. Since the space of global sections of $\underline{\sBr}^\dagger_X$ can be identified with $\sBr^\dagger(X)$, we have the following:
\end{pg}

\begin{lemma}
	Let $X$ be a $0$-truncated quasi-geometric stack. Then the homotopy groups of $\sBr^\dagger(X)$ are given by
$$
\pi_n\sBr^\dagger(X)\simeq
\begin{cases}
\HH^2_{\fpqc-\et}(X, \calO_X^\times)\times \HH^1_{\fpqc-\et}(X, \underline{\mathbb{Z}}) & \text{if $n=0$} \\
\HH^1_{\fpqc-\et}(X, \calO_X^\times)\times \HH^0_{\fpqc-\et}(X, \underline{\mathbb{Z}})  & \text{if $n=1$} \\
\HH^0_{\fpqc-\et}(X, \calO_X^\times) & \text{if $n=2$} \\
0 & \text{if $n\geq 3$,} 
\end{cases}
$$
where $\HH^\ast_{\fpqc-\et}(X,\bullet)$ denotes of the cohomology group of the fpqc-\'etale $\infty$-topos of $X$. 
\end{lemma}

\begin{remark}
	In the special case where $X$ is an ordinary quasi-compact quasi-separated scheme, we recover \cite[7.14]{MR3190610}. 
\end{remark}

\bibliography{chough_brauer}

\providecommand{\bysame}{\leavevmode\hbox to3em{\hrulefill}\thinspace}
\providecommand{\MR}{\relax\ifhmode\unskip\space\fi MR }
\providecommand{\MRhref}[2]{%
  \href{http://www.ams.org/mathscinet-getitem?mr=#1}{#2}
}
\providecommand{\href}[2]{#2}
\begin{thebibliography}{10}

\bibitem{MR3190610}
Benjamin Antieau and David Gepner, \emph{Brauer groups and \'{e}tale cohomology
  in derived algebraic geometry}, Geom. Topol. \textbf{18} (2014), no.~2,
  1149--1244. \MR{3190610}

\bibitem{MR0260746}
M.~Artin, \emph{Algebraization of formal moduli. {I}}, Global {A}nalysis
  ({P}apers in {H}onor of {K}. {K}odaira), Univ. Tokyo Press, Tokyo, 1969,
  pp.~21--71. \MR{0260746}

\bibitem{deJong}
Aise~Johan de~Jong, \emph{A result of gabber}, Preprint.

\bibitem{MR0217086}
A.~Grothendieck, \emph{\'{E}l\'{e}ments de g\'{e}om\'{e}trie alg\'{e}brique.
  {IV}. \'{E}tude locale des sch\'{e}mas et des morphismes de sch\'{e}mas.
  {III}}, Inst. Hautes \'{E}tudes Sci. Publ. Math. (1966), no.~28, 255.
  \MR{0217086}

\bibitem{MR0238860}
\bysame, \emph{\'{E}l\'{e}ments de g\'{e}om\'{e}trie alg\'{e}brique. {IV}.
  \'{E}tude locale des sch\'{e}mas et des morphismes de sch\'{e}mas {IV}},
  Inst. Hautes \'{E}tudes Sci. Publ. Math. (1967), no.~32, 361. \MR{0238860}

\bibitem{MR1608798}
Alexander Grothendieck, \emph{Le groupe de {B}rauer. {I}. {A}lg\`ebres
  d'{A}zumaya et interpr\'{e}tations diverses [ {MR}0244269 (39 \#5586a)]},
  S\'{e}minaire {B}ourbaki, {V}ol. 9, Soc. Math. France, Paris, 1995, pp.~Exp.
  No. 290, 199--219. \MR{1608798}

\bibitem{MR1608805}
\bysame, \emph{Le groupe de {B}rauer. {II}. {T}h\'{e}orie cohomologique [
  {MR}0244270 (39 \#5586b)]}, S\'{e}minaire {B}ourbaki, {V}ol. 9, Soc. Math.
  France, Paris, 1995, pp.~Exp. No. 297, 287--307. \MR{1608805}

\bibitem{MR3705292}
Jack Hall and David Rydh, \emph{Perfect complexes on algebraic stacks}, Compos.
  Math. \textbf{153} (2017), no.~11, 2318--2367. \MR{3705292}

\bibitem{MR3754421}
\bysame, \emph{Addendum to ``\'{E}tale d\'{e}vissage, descent and pushouts of
  stacks'' [{J}. {A}lgebra 331 (1) (2011) 194--223] [ {MR}2774654]}, J. Algebra
  \textbf{498} (2018), 398--412. \MR{3754421}

\bibitem{khanlecture3}
Adeel~A. Khan, \emph{Compact generation of quasi-coherent sheaves}.

\bibitem{MR1771927}
G\'{e}rard Laumon and Laurent Moret-Bailly, \emph{Champs alg\'{e}briques},
  Ergebnisse der Mathematik und ihrer Grenzgebiete. 3. Folge. A Series of
  Modern Surveys in Mathematics [Results in Mathematics and Related Areas. 3rd
  Series. A Series of Modern Surveys in Mathematics], vol.~39, Springer-Verlag,
  Berlin, 2000. \MR{1771927}

\bibitem{MR2717174}
Jacob Lurie, \emph{Derived algebraic geometry}, ProQuest LLC, Ann Arbor, MI,
  2004, Thesis (Ph.D.)--Massachusetts Institute of Technology. \MR{2717174}

\bibitem{MR2522659}
\bysame, \emph{Higher topos theory}, Annals of Mathematics Studies, vol. 170,
  Princeton University Press, Princeton, NJ, 2009. \MR{2522659}

\bibitem{lurie2018sag}
\bysame, \emph{Spectral algebraic geometry}, Last update: Feb 2018, Preprint.

\bibitem{lurie2011dagvii}
\bysame, \emph{Spectral schemes}, Last update: June 2011, Preprint.

\bibitem{lurie2011dagxi}
\bysame, \emph{Descent theorems}, Last update: Sep 2011, Preprint.

\bibitem{lurie2017ha}
\bysame, \emph{Higher algebra}, Last update: Sep 2017, Preprint.

\bibitem{stacks-project}
The {Stacks Project Authors}, \emph{\textit{Stacks Project}},
  \url{https://stacks.math.columbia.edu}, 2018.

\bibitem{MR2957304}
Bertrand To\"{e}n, \emph{Derived {A}zumaya algebras and generators for twisted
  derived categories}, Invent. Math. \textbf{189} (2012), no.~3, 581--652.
  \MR{2957304}

\end{thebibliography}
\bibliographystyle{amsplain}

\end{document}